\documentclass[onefignum,onetabnum]{article}

\usepackage[T1]{fontenc}
\usepackage[utf8]{inputenc}
\usepackage{graphicx}
\usepackage{mathtools}

\usepackage{amsthm,amsmath,amssymb}
\usepackage{paralist,enumitem}
\usepackage{verbatim}
\usepackage[right]{eurosym} 
\usepackage{url}
\usepackage{makecell}
\usepackage{booktabs}
\usepackage[section]{placeins}

\usepackage{algorithm2e}
\usepackage{algorithmic} 

\usepackage{float}
\usepackage{subfig}

\usepackage{tikz}
\usetikzlibrary{decorations,decorations.pathmorphing,decorations.markings,decorations.text,%
decorations.footprints,decorations.fractals,decorations.pathreplacing,decorations.shapes,backgrounds,fit,%
trees,shapes.geometric,calc,shapes.symbols,patterns,positioning,shadows}
\tikzset{vertex/.style={circle,draw=blue!50,fill=blue!15,thick,minimum size=5mm,inner sep=0.5mm}} 
\tikzset{block/.style={vertex,rectangle}}
\tikzset{squigg/.style=decorate,decoration={snake,amplitude=0.5mm,segment length=1.5mm}} 
\tikzset{dreieck/.style={shape=isosceles triangle,%
  isosceles triangle apex angle=60,shape border rotate=90,fill=black!15,draw=black,thick,minimum width=1cm}}
\tikzset{>=stealth}
\usepackage{xcolor}

\newtheorem{theorem}{Theorem}
\newtheorem{remark}{Remark}
\newtheorem{definition}{Definition}

\newcommand{\R} {\mathbb{R}}

\newcommand{\dd}{\,\mathrm{d}}
\newcommand{\Wad}{W_\text{ad}}
\newcommand{\calL}{\mathcal L} 
\newcommand{\calH}{\mathcal H}

\newcommand{\calJ}{\mathcal J}

\newcommand{\cost}{c}
\newcommand{\lb}{x^{\ell}}
\newcommand{\ub}{x^u}
\newcommand{\zin}{\hat{z}}
\newcommand{\ux}{u^x}

\title{A new perspective on dynamic network flow problems via 
	port-Hamiltonian systems}
\author{%
Onur~Tanil~Doganay%
\thanks{University of Wuppertal,
        Institute of Mathematical Modelling, Analysis and Computational Mathematics
        ({doganay@uni-wuppertal.de},
        {klamroth@uni-wuppertal.de},
        {lang@uni-wuppertal.de},
        {stiglmayr@uni-wuppertal.de},
        {totzeck@uni-wuppertal.de})
        }%
\and
Kathrin~Klamroth%
\footnotemark[1]%
\and
Bruno~Lang%
\footnotemark[1]%
\and
Michael~Stiglmayr%
\footnotemark[1]%
\and
Claudia~Totzeck%
\footnotemark[1]%
}
\date{}

\begin{document}

\maketitle
\begin{abstract}
  We suggest a global perspective on dynamic network flow problems that takes advantage of the similarities to port-Hamiltonian dynamics. Dynamic minimum cost flow problems are formulated as open-loop optimal control problems for general port-Hamiltonian systems with possibly state-dependent system matrices. We prove well-posedness of these systems and characterize optimal controls by the first-order optimality system, which is the starting point for the derivation of an adjoint-based gradient descent algorithm. Our theoretical analysis is complemented by a proof of concept, where we apply the proposed algorithm to static minimum cost flow problems and dynamic minimum cost flow problems on a simple directed acyclic graph. We present numerical results to validate the approach. 
\end{abstract}

\textbf{keywords:} 
port Hamiltonian systems $\bullet$ ordinary differential equations $\bullet$ optimal control $\bullet$ minimum cost network flow $\bullet$ dynamic network flow.

\textbf{AMS}:
93B70 $\bullet$ 
49J15 

\section{Introduction}\label{sec:intro}
Network flow problems occur in a variety of applications, including routing problems, supply chain management applications, and combinatorial optimization problems like matching and assignment problems. We refer to \cite{ahuja93network} for a comprehensive introduction. 
In this paper, we focus on minimum cost flow problems. 
In the static and linear case, minimum cost flow problems seek to send a
constant amount of flow from a finite set of supply nodes to a finite set of demand nodes through a directed graph at minimum cost, while obeying the flow conservation constraints. The edges of the graph usually have some upper and lower capacities limiting the amount of flow, and each unit of flow induces a fixed cost on every edge. 
Static linear minimum cost flow problems \eqref{eq:MCNF} have a simple linear programming formulation
\begin{equation}\tag{MCFP}\label{eq:MCNF}
	\begin{array}{rr@{\extracolsep{0.75ex}}c@{\extracolsep{0.75ex}}l}
		\min &\multicolumn{3}{l}{\displaystyle  \cost^\top x} \\ 
		\text{s.\,t.} & A\,x &=& b\\
		& \lb &\leq &x \leq \ub.
	\end{array}
\end{equation}
Here, $x$ is the vector of (unknown) flows on the edges, $c$ is the vector of flow costs (per unit of flow), $A$ is the node arc incidence matrix of the underlying graph, $b$ is the vector of supplies and demands at the nodes, and $x^\ell$ and $x^u$ are lower and upper flow bound constraints, respectively.
Static minimum cost flow problems can be solved efficiently, for example, by the network simplex algorithm or by specific algorithms tailored, e.g., for sparse or dense networks. We refer again to \cite{ahuja93network} for a comprehensive introduction to the field, and to \cite{kova:mini:2015} for a numerical comparison of a variety of solution methods. 

In many practical applications, supplies and demands vary over time, hence requiring dynamic network flow models. These can also be characterized by a set of flow conservation constraints, that typically have to be satisfied at all points in time.  See, for example, \cite{aron:asur:1989,kotnyek03annotated,pageni21survey,pyakurel20network,skutella09introduction} for introductory surveys on dynamic network flow problems and for an array of potential applications.

An established way to handle flows over time is based on time-discretization and on the construction of so-called time-expanded networks, see, for example,  \cite{fleischer07quickest,Gross14,skutella09introduction}. By generating a copy of the (static) network for each time step and introducing arcs that reflect travel times between nodes by spanning over different time steps, dynamic network flows can be approximated by static network flows that can be computed using classical methods. However, this comes at the cost of a largely increased network complexity -- and hence computational time -- and becomes impractical for large real-world applications. Moreover, it is difficult to incorporate non-linear and flow dependent edge costs into the framework of time-expanded networks. One possible but rather rough approximation of flow dependent edge costs consists in introducing several (parallel) arcs with limited capacity and increasing costs, so that larger flow amounts are forced to use more expensive edges  \cite{koehler05flows}.

Alternatively, dynamic and non-linear network flows can be modelled by using one dimensional partial differential equations (PDE) on each edge and, if necessary, also ordinary differential equations in each node of the network. This is particularly well-suited whenever a high accuracy is needed and was successfully applied for modelling, simulation \cite{SailerMarheineke2022,egger2018stability,egger2018structure,mehrmann2018model} and optimization \cite{EggerKuglerWollner2017,Burlacu2019,Gugat2005optimal,Gugat2018mip,Gugat2018towards,krug2021optimization,Daensches2023adaptive,jaekle2023optimal,herty2010,Goettlich2016} of gas, traffic or heating networks, see also \cite{gugat2022modeling} for a recent overview in the gas setting and \cite{waternetworks} for more details on water networks. Due to the large computational complexity, however, 
these models are often limited to smaller instances, see the discussion in \cite{Cominetti2018}.  

The port-Hamiltonian flows that are considered in this paper are modelled in terms of ordinary differential equations (ODE), and can be viewed as a compromise between (cheap and inaccurate) discretized linear and static flows and (accurate and complex) flow dynamics based on PDE. We will show that port-Hamiltonian flows cover static as well as dynamic models and allow for a wide range of further modelling options through the system matrices and via potential ports. Further, we incorporate capacity constraints, which require tailored optimal control problems for PHS. We propose gradient-based optimization approaches that are then applied to static and dynamic flow settings. 
Characteristic for (static and dynamic) network flow problems are the flow conservation constraints that guarantee that no flow is added or lost at any node and that the (external) supplies and demands are satisfied. This flow conserving property, that is also sometimes referred to as Kirchhoff's law, is the basis of the suggested port-Hamiltonian formulation.  To showcase the relation of \eqref{eq:MCNF} and the port-Hamiltonian framework we introduce the nodes $\rho$ as additional state. Flow conservation in the nodes leads to the constraint $-A^\top\rho =0$. Hence, the reformulation of the minimum cost flow problem in the extended state space is given by
\begin{equation}\tag{PH-MCFP}\label{eq:PH-MCNF}
	\begin{array}{rr@{\extracolsep{0.75ex}}c@{\extracolsep{0.75ex}}l}
		\min &\multicolumn{3}{l}{\displaystyle  \cost^\top x} \\ 
		\text{s.\,t.} & A\,x &=& b \\ &-A^T\,\rho &=&0\\
		& \lb &\leq &x \leq \ub.
	\end{array}
\end{equation}
Now, we can see the skew-symmetric structure typical for port-Hamiltonian systems considering $z=(\rho,x)$ and $J=\left(\begin{smallmatrix} 0 & A \\ -A^\top & 0\end{smallmatrix}\right)$. Indeed, the equality constraints are given by $Jz = Bu$ with $B = \left(\begin{smallmatrix} I & 0 \\ 0 & 0\end{smallmatrix}\right)$ and $u=(b,0)^\top$. The skew symmetric structure inherent to flow conservation constraints in graphs was observed, for example, in [56], and  will be further explored in Section~\ref{sec:staticflow}. 
As we shall see, port-Hamiltonian systems provide a general framework that can be used to model classical (static) flow conservation constraints as well as a wide range of dynamic network flow formulations, including flow dependent costs and constraints. 

Port-Hamiltonian systems (PHS) were introduced in 1992 by Arjan van der Schaft and Bernhard Maschke in order to formalize the power-conserving coupling of dynamical systems on different domains \cite{S2_MavdS92}. By now, the PHS framework for modelling is well-established in the engineering community, in particular, the inherent features of energy conservation, elegant (de)coupling into submodules and PHS preserving control approaches using system theory are appreciated \cite{schaft2014survey}. Initially, the main focus was on physical systems, but the abstract framework of Dirac structures \cite{courant1990dirac} allows for straightforward generalizations attracting the interest of mathematicians.
In recent years, linear PHS with finite and infinite dimensional states became well-understood \cite{jacob2012linear}. However, there are many open problems concerning nonlinear systems, differential-algebraic PH structures, stochastic PHS and optimization of PHS. 

Concerning the optimization of PHS most approaches in the literature use closed-loop controllers, for example, by constructing a feedback
that, in combination with the original system, yields 
again a PHS structure, see for example \cite{breiten2022structure,wu2020reduced}, \cite{koelsch2021optimal} for a variant based on Control-Lyapunov functions and the survey \cite{nageshrao2016control}. Open loop approaches are very recent \cite{faulwasser2022optimal,maschke2022optimal,schaller2021control,karsai2023manifold} and so far focused on minimizing the energy supply, which naturally leads to solutions with turnpike properties. Since our ultimate goal is the solution of dynamic minimum cost flow problems, we focus on the optimal control of general cost functionals constrained by nonlinear ODE-PHS in the following. Towards this end, we introduce a general class of optimal control problems constrained by nonlinear ODE with PHS structure and analyze the well-posedness. In order to derive a gradient-descent algorithm we compute the first-order optimality system that characterizes the optimal control. We then consider dynamic minimum cost flow problems as an application of this general framework. Here, we exploit the intrinsic flow conservation property of PHS and recall the relation of the incidence matrix $A$ of the network and the skew-symmetric system matrix $J$ as in \eqref{eq:PH-MCNF}, which is well-known in the context of PHS on graphs \cite{schaft2014survey}. We then first verify the new modelling perspective at the (well-understood) special case of \emph{static} minimum cost flow problems, before discussing an instance of a dynamic minimum cost flow problem as a proof of concept for the new approach.

This paper is organized as follows. To pave the way for the port-Hamiltonian formulation of network flow problems, we start in Section~\ref{sec:control} with the formulation of a general class of optimal control problems constrained by ODE systems with PHS structure. As we aim to derive a gradient-descent algorithm we assume that the control space admits a Hilbert structure and prove well-posedness of the state system as well as continuous dependence on the data in Carath\'eodory sense. These results are exploited in the proof of well-posedness of the optimal control problem. Then we derive the first-order optimality conditions which lay the ground for the gradient-descent algorithm. In Section~\ref{sec:staticflow} we discuss the relationship of the general PHS setting, especially PHS on graphs, and the special case of minimum cost network flow problems. To verify the PHS approach at a well-understood problem class, we first solve static minimum cost flow problems with the proposed framework. A proof of concept for dynamic minimum cost flow problems is given in Section~\ref{sec:timeNWF}. Then we draw our conclusion.

\section{Optimal control of ODE with PHS structure}\label{sec:control}
In this section we formally introduce a general class of optimal control problems for PHS that will be central for the PHS-based formulation of dynamic network flow problems. Since most of the results follow from more or less standard arguments, we have moved the technical parts of the derivation to the appendix. For a comprehensive introduction to PHS see, for example, \cite{schaft2014survey}.

Let us assume to have a Hamiltonian $\calH \colon \R^n \rightarrow \R$ of the form
\begin{equation}
	\calH(z(t)) = \frac{1}{2} z(t)^\top Q \,z(t), \qquad Q \in \R^{n\times n},\; Q>0
\end{equation}
and the corresponding port-Hamiltonian system given by
\begin{subequations}\label{eq:phs}
	\begin{align}
		\frac{\dd}{\dd t} z &= (J(z)-R(z))\, Q\, z + B(z)\, u, \qquad z(0) = \zin, \label{eq:PHSstate}\\
		y &= B(z)^\top Q \, z , \label{eq:PHSoutput}
	\end{align}
\end{subequations}
where $J$ and $R$ are locally Lipschitz, in more detail, $J,R \in \text{Lip}_\text{loc}(\R^n, \R^{n\times n})$, with $ J(z)^\top = -J(z)$ for all $z \in \R^n$, $R(z) \ge 0$ for all $z \in \R^n$ and the input matrix $B\in \text{Lip}_\text{loc}(\R^n, \R^{n\times m})$. Note that the output equation \eqref{eq:PHSoutput} is passive, in the sense that $y$ can be easily computed once we have $z$.

\subsection{The optimal control problem}
We consider the task of finding an initial condition $\zin$ or input $u \in L^2(0,T;\R^m)$ (or both) that allows us to drive the dynamics either as close as possible to a desired state $z_\text{des} \in H^1(0,T;\R^n)$ 
or to find a dynamic that satisfies, for example, given supplies and demands at minimum cost in a network flow sense.
To this end, we let $w=(u,\zin)$ and propose an optimal control problem given by
\begin{equation}\label{eq:opt_problem}\tag{P}
	\min\limits_{(z,w) \in Z\times \Wad} \calJ(z,w) \qquad \text{subject to the dynamics }  \eqref{eq:phs},
\end{equation}
where
\[
\calJ(z,w):=\int_0^T \cost(z(t), z_\text{des}(t)) \dd t +\cost_T(z(T),z_{\text{des}}(T)) + \frac{\lambda}{2}\, \| w \|^2 
\]
with cost functions $c,c_T:\R^n\times\R^n\rightarrow\R$. $Z$ is the state space (see below for a specification), and $\Wad$ is the set of admissible controls, which needs to be specified individually for each problem at hand. The first two terms of the cost functional allow us to define the main objective of the problem by penalizing undesired states over
the entire time interval $[ 0, T ]$ or at the final time $T$, while the third term is optional. It is needed in case the set of admissible controls is unbounded.

For the derivation of the first order optimality system and the gradient, it is useful to have Riesz representation theorem, which requires controls from a Hilbert space. We therefore choose $u \in L^2(0,T;\R^m)$ in the following and consider solutions to the state system in Carath\'eodory sense.

Let $u \in L^2(0,T;\R^m)$ and $\zin \in \R^n$. Then the right-hand side of \eqref{eq:PHSstate} given by
$$f(t,z) = (J(z) - R(z))\,Q\, z + B(z)\, u(t)$$
is defined on the rectangle $\Omega = \{ (t,z) \colon 0 \leq t \leq T, |z-\zin| \leq b \}$ for $b>0$. Moreover, it is a Carath\'eodory function, i.e., it holds
\begin{enumerate}[label={(\alph*)}] 
	\item $f(t,z)$ is continuous in $z$ for each fixed $t$, 
	\item $f(t,z)$ is measurable in $t$ for each fixed $z$,
	\item there is a Lebesgue-integrable function $m_\Omega \colon [0,T] \rightarrow [0,\infty)$ such that $$|f(t,z)| \le m_\Omega(t) \text{ for all } (t,z) \in \Omega.$$ 
\end{enumerate}
Indeed, (a) holds by the assumptions on $J,R,Q$ and $B$; and (b) holds by the assumption on $u.$ To show (c) we note that it holds
\[
|f(t,z)| \le \max\limits_{z \colon |z-\zin| \le b } \Big( \| J(z) \| + \|R(z)\| \Big) \|Q\|\,  |z| + \max\limits_{z \colon |z-\zin| \le b } \| B(z) \| \,|u(t)| =: m_\Omega(t). 
\]
Furthermore, the assumptions on the system matrices yield the existence of a function $k_\Omega(t)$ such that
\[
|f(t,z) - f(t,\tilde{z})| \leq k_\Omega(t) \, |z - \tilde{z}|, \qquad (t,z),(t,\tilde{z}) \in \Omega.
\]
Altogether we obtain the well-posedness of the state equation, see \cite[Theorem 5.3]{hale2009ordinary} for details.
\begin{theorem}\label{thm:existence_state}
	Let $J,R,Q,B$ and $u$ as above. Then for any $\zin$ there exists a unique absolutely continuous solution $z$ satisfying \eqref{eq:phs} except on a set of Lebesgue measure zero and $z(0)=\zin$.
\end{theorem} 
In particular, Theorem~\ref{thm:existence_state} yields $z\in H^1(0,T;\R^n)$. Hence, we define the state space $Z = H^1(0,T;\R^n)$. Let $W := L^2(0,T;\R^m) \times \R^n$, then we define the
\textit{control-to-state map}
\[
S \colon W \rightarrow Z, \qquad w=(u,\zin) \mapsto z.
\]
Furthermore, we use $S$ to define the reduced cost functional
\[
\hat \calJ(w) := \calJ(S(w),w).
\]
We emphasize that using $\hat \calJ$ we treat the state constraint of the optimization problem \eqref{eq:opt_problem} implicitly. 
In fact, we will derive an optimization algorithm based on the reduced form and only update the controls $w$ to drive the dynamics to reduce the objective. 

The next theorem shows the boundedness of the states with respect to the controls.
\begin{theorem}\label{thm:boundedness_state}
	Let $J,R,Q,B$ and $u$ as above. Then the state solution $z$ to \eqref{eq:phs} is bounded by the control $w$. In more detail, there exists a constant $C>0$ such that
	\[
	\| z \|_{H^1(0,T;\R^n)} \leq C \,\| w \|_W.
	\]
\end{theorem}
The proof of Theorem~\ref{thm:boundedness_state} is provided in Appendix~\ref{app:boundedness_state}.

Before we begin with the theorem on the existence of optimal controls, we note that $\lambda > 0$ implies that the cost functional is coercive with respect to $w$ and thus a minimizing sequence is bounded. Moreover, we introduce the state operator $$e \colon H^1(0,T;\R^n) \times W \rightarrow X^* \times \R^n$$ implicitly as 
\begin{align*}
	\big\langle e(z,w), (\begin{smallmatrix}\varphi\\\varphi_0\end{smallmatrix}) \big\rangle_{X^*\times \R^n,X\times \R^n} &= \int_0^T \Big( \frac{\dd}{\dd t} z - (J(z)-R(z))\, Q\, z - B(z)\, u \Big) \cdot \varphi \dd t \\ &\qquad \qquad + (z(0) - \zin ) \cdot \varphi_0
\end{align*}
where $ \varphi \in X=L^2(0,T;\R^n)$ and \(\varphi_0\in\R^n\). 

To state the result on the existence of an optimal control, we introduce the notion of weak continuity.
\begin{definition}
	Let $W,Z$ be Banach spaces. A mapping $A \colon W \rightarrow Z$ is \textit{weakly continuous}, if $w_k \rightharpoonup w$ in W implies $A(w_k) \rightharpoonup A(w)$ in $Z$.
\end{definition}
\begin{theorem}\label{thm:optimal_input}
	Let $\lambda>0$ or $\Wad \subset W$ closed and bounded. Further, let $J, R, B$ be weakly continuous and $\cost(z,z_\text{des})$ continuous and convex w.r.t.~$z$ and $\cost_T(z(T),z_\text{des}(T))$  continuous and convex w.r.t.~$z(T)$. Then there exists a solution to \eqref{eq:opt_problem}. 
\end{theorem}
The proof of Theorem~\ref{thm:optimal_input} is given in Appendix~\ref{app:optimal_input}.

In the following, we formally compute the optimality system that allows us to characterize candidates for optimal controls. We remark that more regularity of the system matrices is required to make the following results rigorous.

\subsection{First-order optimality system}

We use a Lagrangian approach to derive the first-order optimality system formally. We therefore introduce the Lagrangian corresponding to \eqref{eq:opt_problem} as
\begin{equation*}
	\calL(z,w,\varphi) = \calJ(z,w) - \big\langle e(z,u), (\begin{smallmatrix}\varphi\\\varphi_0\end{smallmatrix}) \big\rangle_{X^*\times \R^n,X\times \R^n}.
\end{equation*}
The first-order optimality system is now characterized by the solution of $\dd\calL =0$. In more detail,
\[
\dd_\varphi \calL = \frac{\dd}{\dd t} z - (J(z)-R(z))\, Q\, z - B(z)\, u = 0, \qquad \dd_{\varphi_0} \calL = z(0) - \zin = 0,
\]
which is the state equation. For the adjoint equation we compute for an arbitrary direction $h \in Z$
\begin{align*}
	\dd_z\calL(z,w,\varphi) [h] &= \int_0^T \cost'(z,z_\text{des})\cdot h \dd t + \cost_T'(z(T),z_\text{des}(T)) \cdot h(T) - h(0) \cdot \varphi_0 \\& \quad - \int_0^T \frac{\dd}{\dd t} h \cdot \varphi - (J'(z)[h] - R'(z)[h])\, Q\, z \cdot \varphi \dd t  \\&\quad -\int_0^T (J(z)-R(z))\, Q\, h \cdot \varphi - B'(z)[h]\, u \cdot \varphi \dd t 
\end{align*}
Assuming $\varphi \in Z$ we integrate by parts to obtain
\begin{align*}
	\dd_z\calL(z,w,\varphi) [h] &= - h(0) \cdot \varphi_0 - \big[ h \cdot \varphi \big]_0^T \\& \quad +\int_0^T \cost'(z,z_\text{des})\cdot h \dd t + \cost_T'(z(T), z_{\text{des}}(T)) \cdot h(T)  \\& \quad - \int_0^T \Big( -\frac{\dd}{\dd t} \varphi -  (J'(z)^*[\varphi \otimes Qz]  - R'(z)^*[\varphi \otimes Qz]) \\
	&\qquad \qquad \qquad \qquad \qquad- Q^\top(J(z)-R(z))^\top\varphi +  B'(z)^* [\varphi \otimes u]\Big) \cdot h\dd t
\end{align*}
Since $h$ is arbitrary we can choose $h(0) = 0 $ and $h(T)=0$ to  identify the adjoint equation with the help of the fundamental lemma of the calculus of variations
\begin{align*}
	-\frac{\dd}{\dd t} \varphi
	&= \cost'(z,z_\text{des})
	+ ( J'(z)^*[\varphi \otimes Qz] + R'(z)^*[\varphi \otimes Qz])
	\\ & \qquad\qquad\quad\ 
	+ Q^\top(J(z)-R(z))^\top\varphi -  B'(z)^* [\varphi \otimes u], \\[1.5ex]
	\varphi(T) &= \cost_T'(z(T), z_{\text{des}}(T)).
\end{align*}
The optimality condition is an inequality as $\Wad$ is possibly bounded. Hence an optimal control $\bar w$ satisfies
\[
\left\langle \dd_w \calL(z,w,\varphi) , w - \bar w \right\rangle \ge 0 \quad \text{ for all } w \in \Wad. 
\]
Following the same steps as above, we obtain
\begin{equation}\label{eq:opt1}
	\begin{array}{r}
		(\lambda \,\bar \zin + \varphi(0)) \cdot (\zin - \bar \zin) + \int_0^T \lambda \,\bar u(t) + B(\bar z)^\top \varphi \cdot (u(t) - \bar u(t)) \dd t \ge 0 \qquad\quad{} \\[0.5ex]
		\text{ for all } w=(u,\zin,) \in \Wad.
	\end{array}
\end{equation}

To derive the gradient of the reduced cost functional $\hat \calJ(w)$ we first compute the G\^ateaux derivative in direction $h$ which yields
\begin{align*}
	\dd_w \hat \calJ(w)[h] &= \int_0^T \dd_z\calJ(z,w) \cdot S'(w)h + \dd_w\calJ(z,w) \cdot h \dd t \\
	&=\int_0^T \big( S'(w)^* \dd_z\calJ(z,w) + \dd_w\calJ(z,w) \big ) \cdot h \dd t \\
	&= \int_0^T \dd_w \calL (z,w,\varphi) \cdot h \dd t .
\end{align*}
Together with \eqref{eq:opt1} this allows us to identify the gradient as
\[
\nabla \hat \calJ(w) = \begin{pmatrix} \lambda\, u(t) + B(z)^\top \varphi \\ \lambda\, \zin + \varphi(0) \end{pmatrix}.
\]
We employ this to propose a projected gradient descent algorithm for the optimal control problem \cite{hinze09opt}.

\begin{algorithm}
	\caption{Gradient descent algorithm for \eqref{eq:opt_problem}}
	\label{alg:gradientdescent}
	\begin{algorithmic}[1]
		\STATE{\textbf{initialize:} feasible initial guess $w,$ PHS functions  $J(z),R(z),B(z),Q$, \\ and other algorithmic parameters}
		\STATE{solve state problem to get  $S(w)$}
		\STATE{solve adjoint problem to get $\varphi$ for given $S(w)$ and $w$}
		\STATE{identify the gradient and project to $\mathcal W_\text{ad}$ to obtain $P_{\mathcal W_\text{ad}}(\nabla \hat\calJ(w))$}
		\WHILE{$|P_{\mathcal W_\text{ad}}(\nabla \hat\calJ(w))|\geq \epsilon_\text{stop}$}
		\STATE{choose appropriate step size $\sigma$ with Armijo rule}
		\STATE{$w \gets w - \sigma\,P_{\mathcal W_\text{ad}}(\nabla \hat\calJ(w))$}
		\STATE{solve state problem to get  $S(w)$}
		\STATE{solve adjoint problem to get $\varphi$ for given $S(w)$ and $w$}
		\STATE{identify the gradient and project to $\mathcal W_\text{ad}$ to obtain $P_{\mathcal W_\text{ad}}(\nabla \hat\calJ(w))$}
		\ENDWHILE
		\RETURN optimized control $w$
	\end{algorithmic}
\end{algorithm}

In the following, we focus on PHS that model dynamic network flows. One important feature is flow conservation. In particular, this will constrain the set of admissible controls and lead to natural projection operators. Before we go into the details and numerical results for static network flow problems, we establish the relationship of the general optimal control for ODEs with PH structure and static and dynamic network flow problems.

\section{Port-Hamiltonian formulation of network flows}\label{sec:staticflow}
In this section we start with a reinterpretation of a standard static network flow problem as a special case of a PHS constrained optimal control problem as introduced in Section~\ref{sec:control}. 
While this can not be expected to lead to competitive solution approaches as compared to well-established network optimization algorithms (see, e.g., \cite{ahuja93network}), the PHS perspective offers a wide array of modelling options that go far beyond static network flows. This will be exemplified with time dynamic network flow problems in Section~\ref{sec:timeNWF}. 

Minimum cost flow problems have an array of applications in operations research and thus have been extensively investigated since the 1960s, see, e.g.,  \cite{ahuja93network,chen22maximum,cruz22survey,ford62flows}. While the original formulation considers a static situation and constant flow costs, there are extensions to the dynamic case with flow dependent costs. 
However, linear network flow models require extensive reformulation and linearization techniques to approximate time dynamics and/or non-linearity of flow costs. To avoid these reformulations which generally lead to a dramatic increase of the network size (in terms of the number of nodes and edges) we suggest a port-Hamiltonian formulation of network flows which inherently covers non-linear costs and time-dynamic flows.

Towards this end, let $G=(V,E)$ be a finite directed and connected graph with node set $V=\{1,\dots,N_v\}$ and edge set $E\subset V\times V$, with $|V|=N_v$  and $|E|=N_e$.
An edge $e\in E$ from node $i$ to node $j$, $i\neq j$, is denoted by $e=(i,j)$.  The node arc incidence matrix $A \in \{-1,0,1\}^{N_v \times N_e}$ contains one column for every edge $e=(i,j)\in E$ such that $a_{ie}=1$, $a_{je}=-1$, and $a_{ke}=0$ for all $k\in V\setminus\{i,j\}$. Note that the rows of $A$ sum to the zero vector since each column contains exactly one entry equal to $+1$ and one entry equal to $-1$. Hence, $A$ is not of full rank. Whenever a full-rank matrix is required in the following, we will omit an arbitrary but fixed row of $A$ (e.g., its last row) which yields the so-called full-rank node arc incidence matrix that  still contains the same information as $A$. 

In a static network flow problem we usually assume that finite supplies and demands $b_i\in\R$  are associated with every node $i\in V$, satisfying $\sum_{i\in V} b_i=0$ (i.e., the sum of all demands equals the sum of all supplies), and that non-negative costs $\cost_e\in\R_+$ are associated with all edges $e\in E$. Moreover, all edges $e\in E$ have associated flow bound constraints $\lb_e$ (lower bounds) and $\ub_e$ (upper bounds), with $\lb_e\leq \ub_e$ for all $e\in E$. If not stated otherwise, we assume that $\lb_e=0$ for all $e\in E$. The goal is then to identify a minimum cost flow solution $x\in\R^{N_e}$ that satisfies the flow bound constraints on all edges and the flow conservation constraints in all nodes in the sense that the difference between inflow and outflow at every node equals the supplies and demands at the respective nodes. Comprising all data for supplies and demands, costs, and capacities by $b\in\R^{N_v}$ and  $\cost,\lb,\ub\in\R_+^{N_e}$, respectively, we recall the linear programming formulation for the static and linear \emph{minimum cost flow problem} \eqref{eq:MCNF} from Section~\ref{sec:intro}:
\begin{equation}\tag{MCFP}
	\begin{array}{rr@{\extracolsep{0.75ex}}c@{\extracolsep{0.75ex}}l}
		\min &\multicolumn{3}{l}{\displaystyle  \cost^\top x} \\ 
		\text{s.\,t.} & A\,x &=& b\\
		& \lb &\leq &x \leq \ub.
	\end{array}
\end{equation}
We will assume throughout this paper that \eqref{eq:MCNF} is feasible. An example problem with one supply node (node $1$) and four demand nodes (nodes $2$ to $5$) is illustrated in Figure~\ref{fig:uebung}. 

\begin{figure}[htb]
	\begin{minipage}{0.535\textwidth}
		\centering
		\begin{tikzpicture}[scale=0.9,-latex,>=stealth,shorten >=1pt,auto,node distance=2cm, main node/.style={circle,draw,font=\small\normalfont},
			every label/.style={font=\footnotesize}]
			
			\node[main node, label=above:{\scriptsize\(b_1=4\)}] (1) at (0,3) {$1$};
			\node[main node, label={80:{\scriptsize\(b_2=-1\)}}] (2) at (2.2,1.6) {$2$};
			\node[main node, label=-60:{\scriptsize\(b_3=-1\)}] (3) at (1,0) {$3$};
			\node[main node, label=-120:{\scriptsize\(b_4=-1\)}] (4) at (-1,0) {$4$};
			\node[main node, label=100:{\scriptsize\(b_5=-1\)}] (5) at (-2.2,1.6) {$5$};
			
			\path[every node/.style={font=\scriptsize}]
			(1) edge node [above right] {\((4,1)\)} (2)
			edge node [left] {\((4,4)\)} (4)
			(2) edge node [pos=0.6,right] {\((3,1)\)} (3)
			(3) edge node [right] {\((2,2)\)} (1)
			edge node [below] {\((3,1)\)} (4)
			(4) edge node [pos=0.4,left] {\((3,1)\)} (5)
			(5) edge node [above left] {\((1,1)\)} (1);
		\end{tikzpicture}
	\end{minipage}
	\begin{minipage}{0.46\textwidth}
		\small\setlength\arraycolsep{3pt}
		\[
		A = \begin{pmatrix} 1 & 1 & 0 & -1 & 0 & 0 & -1 \\
			-1 &  0 & 1 & 0 & 0 & 0 & 0 \\
			0 & 0 & -1 & 1 & 1 & 0 & 0 \\
			0 & -1 & 0 & 0 & -1 & 1 & 0 \\
			0 & 0 & 0 & 0 & 0 & -1 & 1 \end{pmatrix}.
		\]
	\end{minipage}
	
	\caption{Illustrative example of a network flow problem with five nodes, i.e.,  $V=\{1,\dots,5\}$, seven edges $E=\{(1,2),(1,4),(2,3),(3,1),(3,4),(4,5),(5,1)\}$, and node arc incidence matrix $A$. The supplies and demands $b_i$, $i=\{1,\dots,5\}$ are indicated next to the nodes, and the upper capacity bounds and costs $(\ub_e,\cost_e)$ are indicated next to the edges. We assume that all lower bounds are equal to zero ($\lb=0$). Note that one possible (non-optimal) feasible flow with cost $20$ is given by $x = (4,2,3,2,0,1,0)^\top$.}
	\label{fig:uebung}
\end{figure}
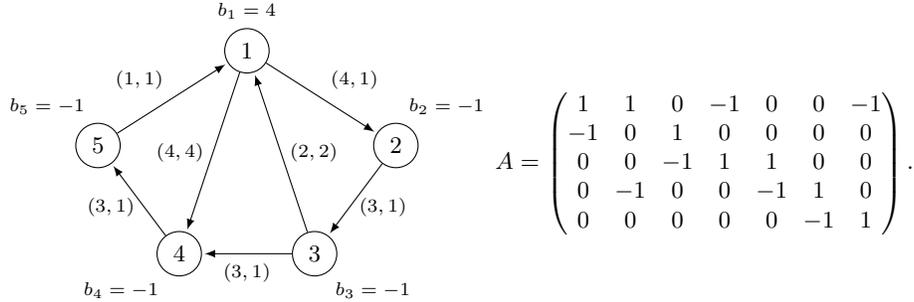

In principle, any optimization algorithm for linear programs can be used to solve \eqref{eq:MCNF}. However, the specific structure of network flow problems can be exploited to significantly improve the computational performance. 
The network simplex algorithm and its variants (see again \cite{ahuja93network}) 
rely on highly efficient updates in the residual network and have become state-of-the-art solution techniques over the years.

\subsection{Static flows as a special case of dynamic flows}\label{sec:statdyn}

In order to cast static network flow problems in the port-Hamiltonian framework \eqref{eq:phs}, we interpret a static flow $\hat{x}\in\R^{N_e}$ as a special case of a dynamic flow $x\in H^1(0,T;\R^{N_e})$, where $x(t)=\hat{x}$ for all $t\in[0,T]$ is desired to be constant over time. Motivated by \cite{schaft13port}, we 
apply the optimal control of port-Hamiltonian systems to the optimal control of (static) network flows, see also \cite{doganay2023modeling}. 
We hence consider flow not only on the edges of the network $G$, but also consider associated potentials, or pressure on the flow, in the nodes of the network. 
Let $\hat{\rho}_v\in\R$ denote a (static) potential (or pressure) in node $v\in V$, and let $\hat{\rho}\in\R^{N_v}$ denote the vector of node potentials. Similar to the flow values on the edges, these could be interpreted as a special case of dynamic node potentials $\rho\in H^1(0,T;\R^{N_v})$, where again $\rho(t)=\hat{\rho}$ for all $t\in[0,T]$ is desired in the static case. The flow conservation constraints (over time) can then be generalized to the system 
\begin{subequations}\label{eq:dynamic_flowconservation}
	\begin{align}
		&\frac{\dd}{\dd t} \rho = A \, x + u^\rho, &\rho(0) = \hat \rho, \label{eq:dynamic_reservoirs}\\
		&\frac{\dd}{\dd t} x = -A^\top \rho , &x(0) = \hat x, \label{eq:dynamic_edges}
	\end{align} 
\end{subequations}
where $u^\rho\in H^1(0,T;\R^{N_v})$ denotes external inputs (i.e., supplies and demands) at the nodes. Note that \eqref{eq:dynamic_reservoirs} ensures that the rate of change of the potential in any node $v$ equals the sum of in- and outflows on the incident edges, possibly plus some (external) input $u^\rho_v$. Similarly, \eqref{eq:dynamic_edges} guarantees that the rate at which the flow changes in any edge $e=(i,j)$ is equal to the difference $\rho_j-\rho_i$ of the potentials (or pressure) at the respective end nodes. 
Note that the node potentials $\rho$ from \eqref{eq:dynamic_flowconservation} are somewhat related to the dual variables of \eqref{eq:MCNF} that are also called node potentials in the network flow context (see, e.g., \cite{ahuja93network}). Similarly, potential-based flow networks (see, e.g., \cite{birk:nonl:1956,dewo:theg:2000,rock:netw:1984}) use node potentials to constrain the flow on the edges through physically motivated constraints. 

In the static case, i.e., when $\rho=\hat{\rho}$ and $x=\hat{x}$ are constant over time, we recover the flow conservation constraints $A\, x=b$ of \eqref{eq:MCNF} from \eqref{eq:dynamic_reservoirs}, where we set $u^\rho=-b$. In addition, \eqref{eq:dynamic_edges} then yields for $u^x=0$ that $-A^\top \rho=0$, which is satisfied for $\rho=\hat{\rho}=0$, i.e., with zero potentials in the nodes.

Now set $n=N_v+N_e$ and let $z=(\rho,x)^\top\in H^1(0,T;\R^{n})$. For completeness, we define $u^x\in H^1(0,T;\R^{N_e})$ as possible external inputs (i.e., inflows or leaks) at the edges and set $u=(u^\rho,u^x)^\top\in H^1(0,T;\R^n)$. Then the system \eqref{eq:dynamic_flowconservation} 
can be rewritten as the port-Hamiltonian system
\begin{equation}\label{eq:network_phs}
	\frac{\dd}{\dd t} z =  J(z) \, Q \, z + B(z) \, u,  \qquad z(0) = \hat z, 
\end{equation} 
where 
\begin{equation*}	    
	J(z) = J = \begin{pmatrix} 0 & A \\ -A^\top & 0  \end{pmatrix}, \quad Q=I\in\R^{n\times n}, \quad \text{and} \quad B(z)=B = \begin{pmatrix}I & 0\\ 0 & 0 \end{pmatrix}\in\R^{n\times n}.
\end{equation*} 
This is a special case of the port-Hamiltonian system \eqref{eq:phs} with $R(z)=0\in\R^{n\times n}$. 
When a solution $z$ of the system \eqref{eq:network_phs} is known, then the output of the system can be easily computed as $y = B^\top z = (\rho,0)^\top\in\R^n$.

In the following, we first utilize a linear program based on the KKT conditions of \eqref{eq:dynamic_flowconservation} as the projection applied in Algorithm~\ref{alg:gradientdescent}, which is an equivalent way to formulate the static problem \eqref{eq:MCNF}. Next, we apply a projection on the space of circulations, i.e., \(\{x\in\R^{N_e}\colon A\,x = 0\}\), paired with a barrier term to achieve results that satisfy the (\(\varepsilon\)-relaxed) capacity constraints, i.e.,\(\,  \lb_{e}-\varepsilon\leq x_e \leq \ub_{e}+\varepsilon, \ \varepsilon>0\), to also solve the static problem \eqref{eq:MCNF}. Finally, we move toward dynamic time-dependent problems and investigate a small instance with two different time-dependent cost functions.

\subsection{Static flows using projections and KKT conditions}\label{sec:staticKKT}
We now consider the static case and set $u^\rho=-b$ as discussed above. Note that $u^\rho$ is thus fixed and no longer a control variable. We hence optimize only over the initial condition $\hat{z}=(0,\hat{x})$, where the cost functional of the control problem is given by
$$\calJ(z,w) =  \frac{1}{T} \int_0^T \cost^\top x(t) \dd t = \cost^\top \hat{x} .$$
In the following, we consider $\hat{\calJ}$ as defined in Section~\ref{sec:control}.

As stated above, in our first approach we utilize the KKT conditions of our formulation for the projection of the gradient $\nabla \hat \calJ(w)$. Recall that an optimal control $\bar w$ satisfies
\begin{equation} 
	\left\langle \nabla \hat \calJ(w) , w - \bar w \right\rangle \ge 0 \quad \text{ for all } w \in \Wad, 
	\label{eq:KKTvarineq}   
\end{equation}
where, for simplicity, we consider only the controllable variables $\hat{x}$ and hence set $\Wad=\{\hat{x}\in\R^{N_e} \colon A\hat{x}=b,\,\lb\leq \hat{x}\leq\ub\}$. 
In  an iteration of the gradient descent Algorithm~\ref{alg:gradientdescent}, we are interested in a direction $h=w- \bar{w}$ such that we achieve equality in \eqref{eq:KKTvarineq}. Towards this end, we investigate the linear program
\begin{align}
	\begin{split}
		\min\limits_{h} \quad& \nabla \hat \calJ(w)^\top h \\
		\text{s.t.} \quad &  A\, h =0, \\ 
		&\lb -\hat{x}\leq h \le \ub -\hat{x}.  
		\label{eq:KKTLP}
	\end{split}
\end{align}
Note that the constraints of \eqref{eq:KKTLP} ensure that $h$ is a circulation in the static network flow problem \eqref{eq:MCNF}, i.e., that
in Algorithm~\ref{alg:gradientdescent} the control updated via the projected gradient $P_{\mathcal W_\text{ad}}(\nabla \hat\calJ(w))$ 
is still feasible. Moreover, for an optimal direction $\bar h$ of the linear program \eqref{eq:KKTLP}, i.e., $P_{\mathcal W_\text{ad}}(\nabla \hat\calJ(w)) = \bar h$, it holds that $\bar w = w -\bar h$. Thus, solving \eqref{eq:KKTLP}, i.e., computing the path from $w$ to $\bar w$, is equivalent to solving \eqref{eq:MCNF}, i.e., computing $\bar w$ directly.

We test the presented approach on the same test cases as in \cite{Loebel1996} which are generated via NETGEN, see \cite{Klingman1974NETGENAP}. We acquired the files containing the test instances on the 23rd of November 2022 from the electronic library of \emph{Zuse Institute Berlin}\footnote{http://elib.zib.de/pub/mp-testdata/mincost/netg/index.html} 
which are unfortunately not available anymore but are reproducible via \cite{Loebel1996,Klingman1974NETGENAP}. Furthermore, since the test cases in \cite{Loebel1996} are rather large, we also include three small exemplary networks (ep1-3) on which we test our approach, see, Figure~\ref{fig:ep}. The numerical experiments are implemented in \emph{Python} (version 3.8.15) using \emph{Anaconda} (version 4.12.0). For the numerical integration of the  ordinary differential equations we apply the symplectic Euler scheme, see, e.g., \cite{Hairer2006}. Furthermore, all linear programs are solved with \emph{IBM ILOG CPLEX} (version 20.1.0.0), since we experienced some issues with the open-source linear program solvers of the Python packages ``\emph{networkx}'', which returns correct optimal objective function values but, surprisingly, infeasible solutions, and ``\emph{scipy.optimize}'' where some optimal objective values were incorrect. To compare our approach, we compute relative errors of the optimized solutions of the static minimal cost network flow problem \eqref{eq:MCNF} to the ones resulting from \emph{CPLEX}. We emphasize that for all test cases of \cite{Loebel1996} and ep1-3 only one gradient step is needed to achieve a relative error of $0.0$ with Algorithm~\ref{alg:gradientdescent}. This is due to the fact that \eqref{eq:MCNF} is actually an equivalent reformulation of the static minimal cost network flow problem.

\begin{figure}[htb!]
	\begin{center}
		\subfloat[Exemplary problem 1 (ep1)
		with 8 nodes and 15 edges.
		\label{fig:ep1}]{
			\begin{tikzpicture}[scale=0.8,-latex,>=stealth,shorten >=1pt,auto,node distance=2cm, main node/.style={circle,draw,font=\small\footnotesize,minimum size=18pt,inner sep=2pt},
				every label/.style={font=\footnotesize}]
				\node[main node, label= left:{\footnotesize\(b_1=20\)}] (1) at (0,3) {1};
				\node[main node, label=above:{\footnotesize\(b_2=0\)}] (2) at (2,4.5) {2};
				\node[main node, label=below:{\footnotesize\(b_3=20\)}] (3) at (2,1.5) {3};
				\node[main node, label=above:{\footnotesize\(b_4=0\)}] (4) at (4.4,4.5) {4};
				\node[main node, label=below:{\footnotesize\(b_5=0\)}] (5) at (4.4,1.5) {5};
				\node[main node, label=above:{\footnotesize\(b_6=-10\)}] (6) at (6.8,4.5) {6};
				\node[main node, label=below:{\footnotesize\(b_7=-10\)}] (7) at (6.8,1.5) {7};
				\node[main node, label=right:{\footnotesize\(b_8=-20\)}] (8) at (8.8,3) {8};
				
				\path[thick,every node/.style={font=\footnotesize}]
				(1) edge node [above,sloped] {\(c_{(1,2)}=2\)} (2)
				(1) edge node [below,sloped] {\(5\)} (3)
				(2) edge node [above,sloped] {\(1\)} (4)
				(2) edge node [pos=0.25,above,sloped] {\(2\)} (5)
				(3) edge node [above,sloped] {\(2\)} (2)
				(3) edge node [pos=0.25,above,sloped] {\(2\)} (4)
				(3) edge node [below] {\(2\)} (5)
				(4) edge node [above,sloped] {\(2\)} (6)
				(4) edge node [pos=0.25,below,sloped] {\(3\)} (7)
				(4) edge node [pos=0.25,below,sloped] {\(4\)} (8)
				(5) edge node [pos=0.25,above,sloped] {\(3\)} (6)
				(5) edge node [below,sloped] {\(1\)} (7)
				(5) edge node [pos=0.25,above,sloped] {\(4\)} (8)
				(6) edge node [above,sloped] {\(3\)} (8)
				(7) edge node [below,sloped] {\(3\)} (8);
			\end{tikzpicture}   
		}
		\vspace{1cm}
		\subfloat[Exemplary problem 2 (ep2)
		with 6 nodes and 9 edges.
		\label{fig:ep2}
		]{
			\begin{tikzpicture}[scale=0.8,-latex,>=stealth,shorten >=1pt,auto,node distance=2cm, main node/.style={circle,draw,font=\small\footnotesize,minimum size=18pt,inner sep=2pt},
				every label/.style={font=\footnotesize}]
				\node[main node, label= left:{\footnotesize\(b_1=30\)}] (1) at (0,3) {1};
				\node[main node, label=above:{\footnotesize\(b_2=0\)}] (2) at (2,4.5) {2};
				\node[main node, label=below:{\footnotesize\(b_3=0\)}] (3) at (2,1.5) {3};
				\node[main node, label=above:{\footnotesize\(b_4=-10\)}] (4) at (4.4,4.5) {4};
				\node[main node, label=below:{\footnotesize\(b_5=-10\)}] (5) at (4.4,1.5) {5};
				\node[main node, label=right:{\footnotesize\(b_6=-10\)}] (6) at (6.4,3) {6};
				
				\path[thick,every node/.style={font=\footnotesize}]
				(1) edge node [above,sloped] {\(c_{(1,2)}=4\)} (2)
				(1) edge node [below,sloped] {\(3\)} (3)
				(2) edge node [above,sloped] {\(2\)} (3)
				(2) edge node [above,sloped] {\(2\)} (4)
				(3) edge node [above,sloped] {\(2\)} (4)
				(3) edge node [below] {\(3\)} (5)
				(3) edge node [above,sloped] {\(5\)} (6)
				(4) edge node [above,sloped] {\(5\)} (6)
				(5) edge node [below,sloped] {\(2\)} (6);
			\end{tikzpicture} 
			
		}
		\vspace{1cm}
		\subfloat[Exemplary problem 3 (ep3)
		with 10 nodes and 20 edges.
		\label{fig:ep3}
		]{
			\begin{tikzpicture}[scale=0.8,-latex,>=stealth,shorten >=1pt,auto,node distance=2cm, main node/.style={circle,draw,font=\small\footnotesize,minimum size=18pt,inner sep=2pt},
				every label/.style={font=\footnotesize}]
				\node[main node, label= left:{\footnotesize\(b_1=20\)}] (1) at (0,3) {1};
				\node[main node, label=above:{\footnotesize\(b_2=10\)}] (2) at (2,4.5) {2};
				\node[main node, label=below:{\footnotesize\(b_3=10\)}] (3) at (2,1.5) {3};
				\node[main node, label=above:{\footnotesize\(b_4=0\)}] (4) at (4.4,4.5) {4};
				\node[main node, label=below:{\footnotesize\(b_5=0\)}] (5) at (4.4,1.5) {5};
				\node[main node, label=above:{\footnotesize\(b_6=-5\)}] (6) at (6.8,4.5) {6};
				\node[main node, label=below:{\footnotesize\(b_7=0\)}] (7) at (6.8,1.5) {7};
				\node[main node, label=above:{\footnotesize\(b_8=-5\)}] (8) at (9.2,4.5) {8};
				\node[main node, label=below:{\footnotesize\(b_9=0\)}] (9) at (9.2,1.5) {9};
				\node[main node, label=right:{\footnotesize\(b_{10}=-30\)}] (10) at (11.2,3) {10};

				\path[thick,every node/.style={font=\footnotesize}]
				(1) edge node [above,sloped] {\(c_{(1,2)}=4\)} (2)
				(1) edge node [below,sloped] {\(6\)} (3)
				(2) edge node [pos=0.25,below,sloped] {\(2\)} (5)
				(2) edge[bend left=38] node [above,sloped] {\(3\)} (6)
				(2) edge node [pos=0.15,above,sloped] {\(5\)} (7)
				(3) edge node [pos=0.25,above,sloped] {\(4\)} (4)
				(3) edge node [below] {\(2\)} (5)
				(3) edge node [pos=0.15,below,sloped] {\(6\)} (6)
				(3) edge[bend right=35] node [below,sloped] {\(3\)} (9)
				(4) edge node [above,sloped] {\(1\)} (6)
				(4) edge node [pos=0.15,above,sloped] {\(4\)} (9)
				(5) edge node [pos=0.15,below,sloped] {\(3\)} (4)
				(5) edge node [below,sloped] {\(3\)} (7)
				(5) edge node [pos=0.1,below,sloped] {\(5\)} (8)
				(6) edge node [above,sloped] {\(3\)} (8)
				(6) edge node [pos=0.7,above,sloped] {\(4\)} (10)
				(7) edge node [pos=0.25,above,sloped] {\(3\)} (6)
				(7) edge node [pos=0.15,below,sloped] {\(3\)} (8)  
				(9) edge node [pos=0.25,above,sloped] {\(3\)} (6)
				(9) edge node [below,sloped] {\(5\)} (10);
			\end{tikzpicture} 
			
		}
	\end{center}		
	\caption{Exemplary minimum cost network flow problems (ep1-3).
		In all cases we assume $\lb_e = 0$ and $\ub_e = 20$ for each edge,
		and the values next to the edges denote the cost of the edges.}%
	\label{fig:ep}
\end{figure}
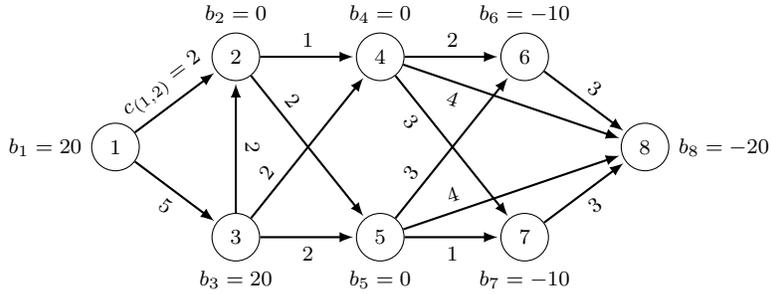
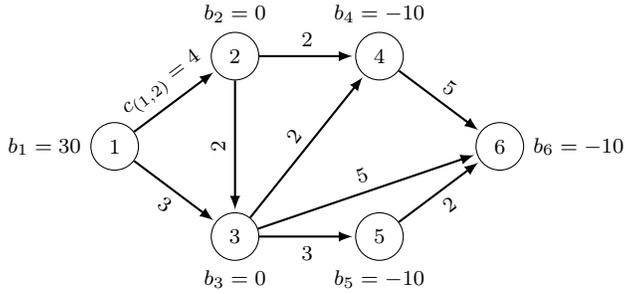
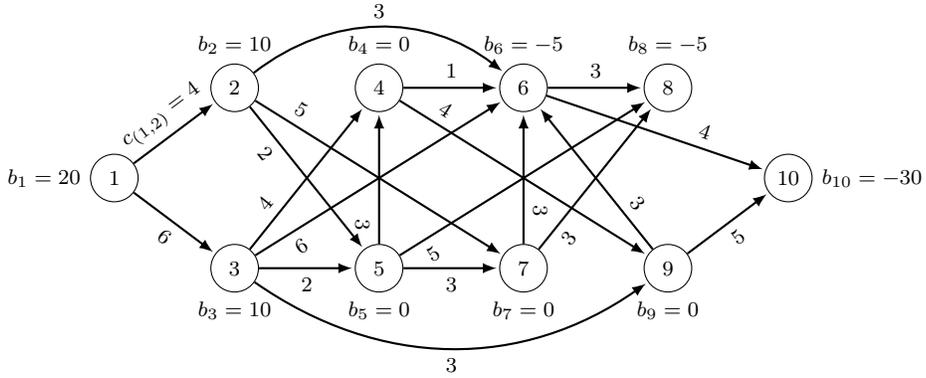

\subsection{Static flows using orthogonal projections and barrier terms}\label{sec:barrier}
We use the same problem formulation for static network flows as in Section~\ref{sec:staticKKT}. 
Instead of including the capacity constraints into the projection, one can use penalty or barrier methods to obtain feasible flows. In this case, a projection is only required to map the flows to the affine-linear space of flows satisfying the flow conservation constraints, i.e., \(\{x\in\R^{N_e}\colon A\,x = b\}\). Penalty and barrier methods both relax (some or all) constraints and penalize constraint violations in the objective function. In the penalty method a penalization term is added to the objective function that is zero if the constraints are satisfied and goes to infinity for increasing constraint violation. On the other hand, in the barrier method a barrier term is added to the objective function that tends to infinity when approaching the boundary of feasible set from the inside.
A comprehensive discussion of penalty and barrier methods is, for example, given in \cite{nocedal06numerical}.
To satisfy the flow conservation constraints after a gradient step, we project the gradient $\nabla \hat \calJ(w)$ to the space of circulations, i.e., \(\{x\in\R^{N_e}\colon A\,x = 0\}\). Towards this end, we utilize the following orthogonal projection, see, e.g., \cite{Petersen_2012},
\begin{equation}
	P_{\mathcal W_\text{ad}}(\nabla \hat\calJ(w)) = \left(I-A^\top(A\,A^\top)^{-1}A\right)\nabla \hat\calJ(w),
	\label{eq:LAproj}
\end{equation}
where $I$ is the identity matrix of the same dimension as $A^\top\,A$. Furthermore, in order to achieve results that do not violate the ($\varepsilon$-relaxed) capacity constraints we use logarithmic functions as a barrier term in our numerical tests, i.e.,
\[
\Theta(x) :=\Theta(x, \alpha,  \varepsilon) := - \alpha \left(\sum_{e\in E} \left(\ln\left(\left(\ub_e-x_e + \varepsilon\right)\right) + \ln\left(\left(x_e
-\lb_e+ \varepsilon\right)\right)\right)\right) ,
\]
where $\alpha>0$ is a penalization parameter and $\epsilon>0$ is a small constant. If the flow value $x_e$ on an edge $e\in E$ reaches one of the flow bounds $\ub_e$ or $\lb_e$ then $\Theta(x)$ becomes large.
Since the feasible set of network flow problems has a specific polyhedral structure, optimal flow solutions often have flow values that are on the boundary of the feasible set, i.e., that satisfy flow bound constraints with equality on many edges. This is also the case for standard choices for the starting solution. 
Therefore, we relax the flow bound constraints to $\lb-\varepsilon \leq x \leq \ub+\varepsilon$ in the barrier term $\Theta(x)$ since otherwise already the starting solution would lead to unbounded cost values.
The cost functional including the barrier term now reads
\[
\calJ_\Theta(z,w):=
\calJ(z,w)+\int_0^T\Theta(x(t))\dd t.
\]
For the implementation of Algorithm~\ref{alg:gradientdescent} we normalize the cost vector $c=\left(c_e\right)_{e\in E}$ with the maximum cost of the edges, i.e., $\bar c=\left(c_e/\max c\right)_{e\in E}$, to ensure that for each test instance the value of $\alpha$ can be chosen from a comparable range. During the optimization process we decrease the values of $\alpha$ and $\varepsilon$ after each gradient step in Algorithm~\ref{alg:gradientdescent}, thus decreasing the influence of the barrier term and the relaxation. We remark that the parameter tuning for this problem is not straightforward since it is unfortunately instance dependent. Further, note that $\Theta$, due to the logarithmic terms, favours flow values in the middle of the intervals $[\lb-\varepsilon, \ub+\varepsilon]$. Thus, this approach has problems computing solutions that are on the boundaries, and hence it is not likely to find optimal flow solutions which are usually extreme points of the polyhedral feasible set. Nevertheless, the combination of the projection \eqref{eq:LAproj} and the barrier term $\Theta$ allows us to extend this approach for dynamic time-dependent network flow problems, as investigated in the following section.

Due to the required instance dependent parameter tuning we narrow our numerical experiment to the exemplary problems (ep1-3). As initial flows for Algorithm~\ref{alg:gradientdescent} we select the optimal solutions of the maximization problem with the same objective function and constraints as \eqref{eq:MCNF}, i.e., the worst case solutions. The results after $300$ iterations and a maximum of $20$ Armijo iterations in Algorithm~\ref{alg:gradientdescent} with a geometric
decrease of $\alpha$ and $\varepsilon$ after each gradient step $k$, where we applied the update scheme $\alpha_{k+1} = 0.9\,\alpha_{k}$
until $\alpha_k=0.01$ is reached and $\varepsilon_{k+1} = 0.99\,\varepsilon_{k}$,
are shown in Table~\ref{tbl:aeLin}. Here, we chose $\sum_{e\in E} |P_{\mathcal W_\text{ad}}(\nabla \hat\calJ(w))_e|<10^{-6}$ as a stopping condition which was not satisfied during the $300$ iterations, i.e., the approaches did not converge and stopped after the maximum number of iterations. However, we observe that for all problems ep1-3 there are no significant changes to the costs of the solutions after approximately \(150\) iterations, see Figure~\ref{fig:Solep}. Furthermore, for all three problems we achieved a relative error of below $5\%$ where graphs with a larger number of edges led to
larger relative errors w.r.t.\ the solutions of \eqref{eq:MCNF}. The reason is that the barrier term $\Theta$ pushes the flow value away from the boundaries towards the middle of the interval $[\lb_{e}-\varepsilon, \ub_{e}+\varepsilon]$ for each edge, and hence for a larger number of edges the deviation from the optimal solution, which is an extreme point, may grow, accumulating in a potentially larger relative error.

\begin{table}[htb]
	\centering
	\begin{tabular}{|l|lll|}
		\hline
		Name    & $\alpha_0$ & $\varepsilon_0$  & \parbox{6em}{Rel.\ Error \\ w.r.t.\ \eqref{eq:MCNF}} \\ \hline
		ep1     & 0.7          & 1.3          & 0.0392                                             \\ \hline
		ep2     & 1.0          & 2.0          & 0.0277                                             \\ \hline
		ep3     & 0.7          & 1.1          & 0.0430                                              \\ \hline          
	\end{tabular}
	\caption{Results after $300$ iterations with $20$ Armijo steps for test instances (ep1-3) with geometrically
		decreasing $\alpha$ and $\varepsilon$.\label{tbl:aeLin}}
\end{table}

\begin{figure}[htb!]
	\centering
	\subfloat[Exemplary problem 1 (ep1) -- cost w.r.t.\ gradient steps\label{fig:Solep1}]{\includegraphics[width=0.45\textwidth]{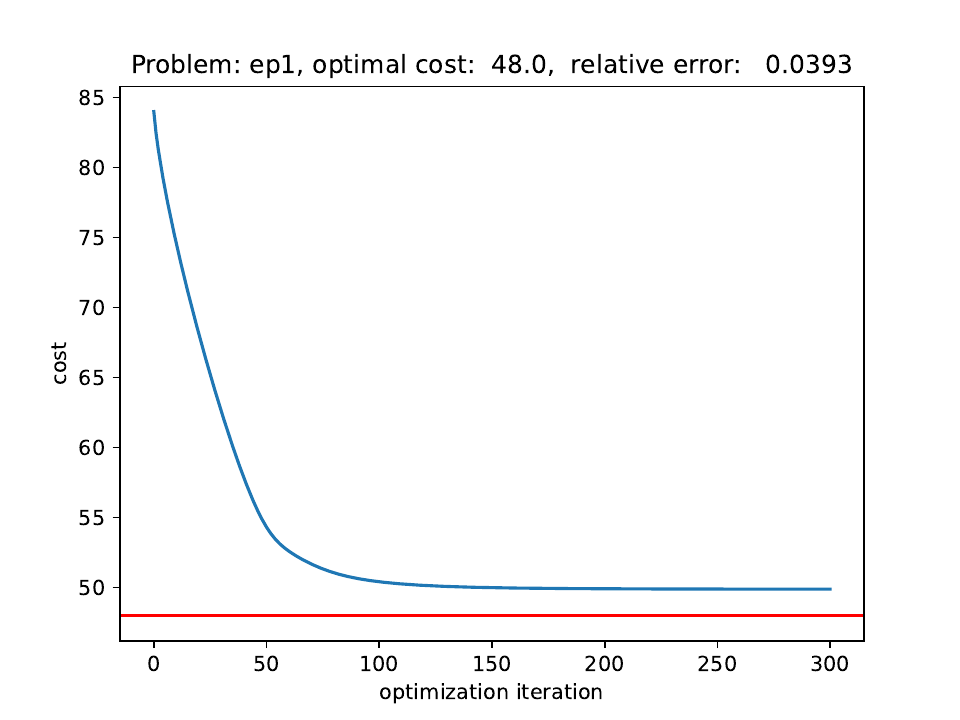} }
	\hspace{\fill}
	\subfloat[Exemplary problem 2 (ep2) -- cost w.r.t.\ gradient steps\label{fig:Solep2}]{\includegraphics[width=0.45\textwidth]{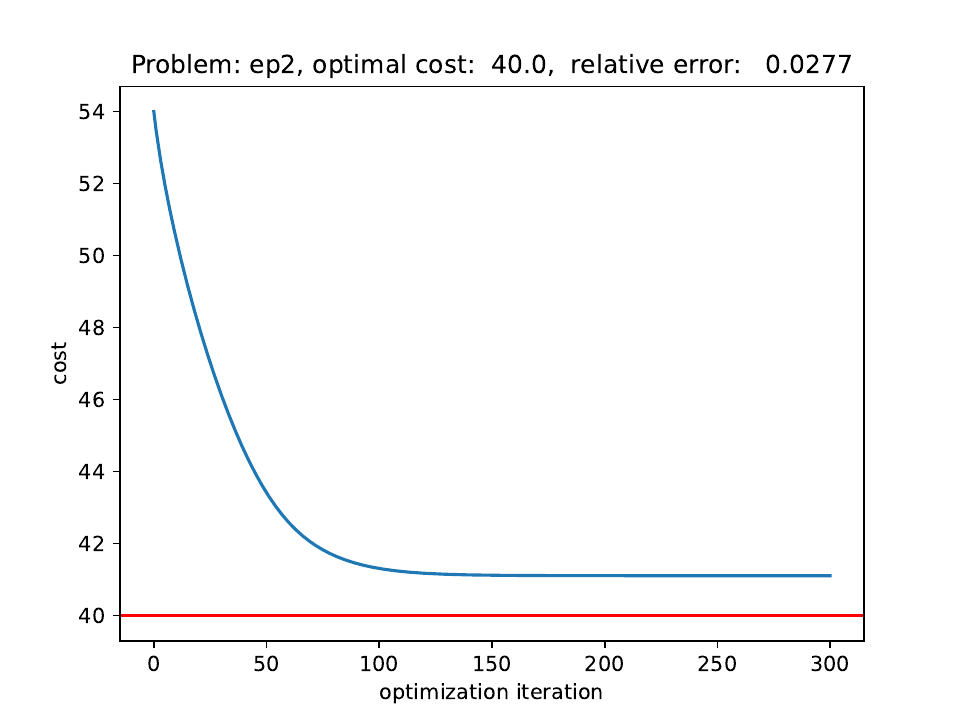} }
	\vspace{\fill}
	\subfloat[Exemplary problem 3 (ep3) -- cost w.r.t.\ gradient steps\label{fig:Solep3}]{\includegraphics[width=0.45\textwidth]{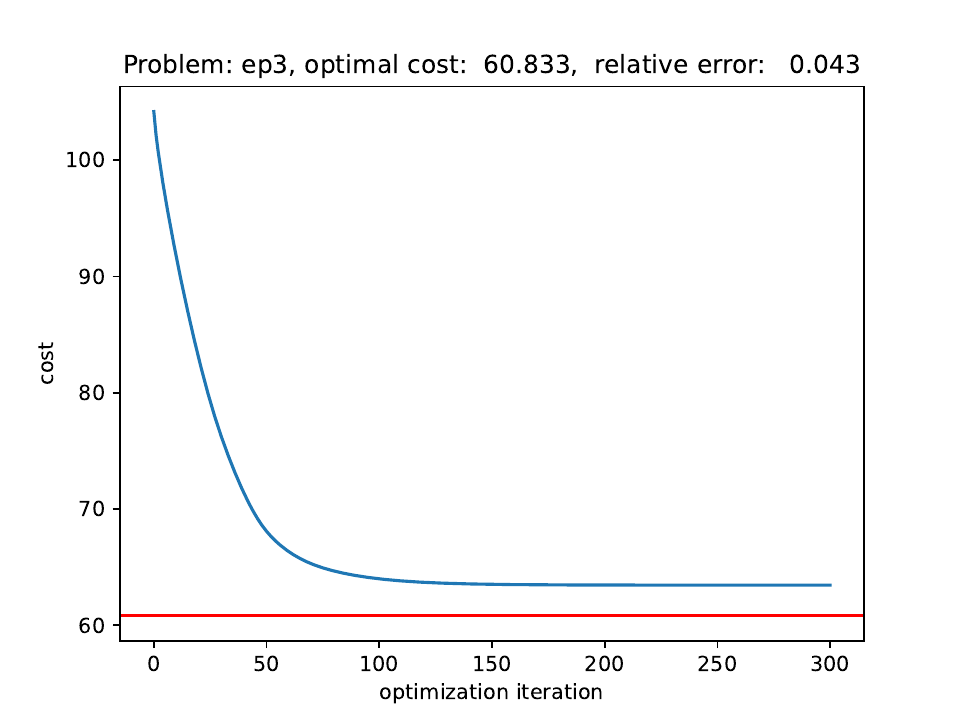} }
	\caption{Cost for $300$ gradient steps with $20$ Armijo iterations for exemplary problems (ep1--3). The red lines show the optimal costs of  \eqref{eq:MCNF} for these problems.}
	\label{fig:Solep}
\end{figure}

\section{Time-dependent network flow problems}\label{sec:timeNWF}
In this section, we extend our approach to dynamic time-dependent problems and show a proof-of-concept with two different time-dependent cost functions. We consider time-dependent network flow problems with relaxed flow bound constraints
\begin{alignat}{3} \label{eq:cost_time}
	\min\quad & \mathcal{J}(z,w), \quad \mathcal{J}(z,w)= \int_0^T \sum_{e \in E} c_e(t)\, x_e(t)   + \frac{\lambda}{2} && \left| \frac{\dd}{\dd t}\ux(t) \right|^2 \dd t \\
	\text{s.t.}\quad
	&\frac{\dd}{\dd t} \rho = A \, x , \qquad \quad &&\rho(0) = \hat \rho, \nonumber \\
	&\frac{\dd}{\dd t} x = -A^\top \rho+  \ux, &&x(0) = \hat x, \nonumber \\
	&\lb-\varepsilon \leq x \le \ub+\varepsilon, \nonumber
\end{alignat}
where, using the notation from \eqref{eq:opt_problem}, $w=(u,\hat{z})$,  $\varepsilon>0$, and where we assume that $u^\rho=0$ and $\hat{z}$ are given and fixed (and not part of the optimization). We hence omit the fixed components $u^\rho$ and $\hat{z}$ of $w$ and identify $w$ with $\ux$ for better readability. Moreover, we consider only the case $\lambda>0$ in the following. Hence, the penalization term in the cost function is active and minimizers have higher regularity. In particular, we allow only to vary the flow on circulations in the network which leads us to the admissible set given by
\[
\Wad = \Bigl\{ \ux \in H^1\left((0,T);\R^{N_e}\right) \colon \quad \ux(0)=0, \quad A\, \ux(t) = 0 \quad\text{a.e. in}\; (0,T) \Bigr\}.
\] 
In addition to ensuring that the flow conservation constraints are satisfied a.e.\ in $(0,T)$, this admissible set guarantees that the control does not change the initial condition $\hat x$. Note that the cost functional penalizes $\|\frac{\dd}{\dd t}\ux\|^2$ which corresponds to the assumption that adjusting the redirection of flow is expensive. 

\begin{remark}
	We note that in $W=\{ \ux \in H^1\left((0,T);\R^{N_e}\right) \colon \ux(0)=0 \}$ the norms $\| \ux\|_\lambda$ defined by 
	\[
	\| \ux\|_\lambda^2 \coloneqq \| \ux\|_{L^2}^2 + \lambda \,\Bigl\| \frac{\dd}{\dd t} \ux \Bigr\|_{L^2}^2
	\]
	and $\| \ux \|_{H^1}$ are equivalent, since for $c_1 = \min(\lambda,1)$ and $c_2 = \max(\lambda,1)$ it holds $c_1 \| \ux \|_{H^1}^2 \le \| \ux \|_{\lambda}^2 \le c_2 \| \ux \|_{H^1}^2$ for each $\ux\in W$. Moreover, since $u^x(0)=0$ is fixed, already the semi-norm $\| \nabla \ux \|_{L^2}^2$ is equivalent to $\| \ux\|_\lambda^2$. Hence,  for $\lambda >0$ the cost function $\calJ(z,w)$ from \eqref{eq:cost_time} is coercive and the result on the existence of optimal controls can be directly applied.
\end{remark}

We use the barrier approach from Section~\ref{sec:barrier} to ensure the $\varepsilon$-relaxed capacity constraints analogous to Section~\ref{sec:barrier}. The derivative of this barrier term appears also in the adjoint system and is given by
\begin{alignat*}{3}
	&-\frac{\dd}{\dd t} \nu = -A^\top \, p + c_\rho(t), \qquad \quad&\qquad \nu(T) = 0, \\
	&-\frac{\dd}{\dd t} p = A \nu + c_x(t) + \nabla \Theta(x(t)), &\qquad p(T) = 0.
\end{alignat*} 

The computation of the gradient is a bit more involved, since we consider $\Wad \subset  H^1\left((0,T);\R^{N_e}\right)$ here. We therefore first compute the linearization of $\hat \calJ(w)$ and then solve an additional problem to identify the gradient using the Riesz representation theorem. 
This leads to 
\[
\dd_{\ux} \calJ((z,\ux))[h] = \int_0^T \lambda \frac{\dd}{\dd t} \ux \cdot \frac{\dd}{\dd t} h + p \cdot h \dd t.
\]
Afterwards, the gradient $g_{\ux}$ is identified by solving 
\[
\int_0^T g_{\ux} \cdot h + \lambda \frac{\dd}{\dd t } g_{\ux} \cdot \frac{\dd}{\dd t}  h\dd t = \int_0^T \lambda \frac{\dd}{\dd t} {\ux} \cdot \frac{\dd}{\dd t} h + p \cdot h \dd t \quad \text{for all $h$}
\]
for $g_{\ux}$. Since the states are only time-dependent, we solve the problem in strong formulation, i.e.,
\[
g_{\ux} - \lambda \frac{\dd^2}{\dd t^2} g_{\ux} = \lambda \frac{\dd^2}{\dd t^2} {\ux}  + p,
\]
with finite differences, where we incorporate the boundary conditions ${\ux}(0)=0$, $\frac{\dd}{\dd t }{\ux}(0) = 0 = \frac{\dd}{\dd t }{\ux}(T)$ in the difference operator. The gradient $g_{\ux}$ is then projected to the feasible set to satisfy the flow conservation constraints.
For the projection of $g_{\ux}(t)$ we again utilize the orthogonal projection
\begin{equation}
	P_{\mathcal W_\text{ad}}\left(g_{\ux}(t)\right)=\left(I-A^\top(A\,A^\top)^{-1}A\right)\,g_{\ux}(t)
\end{equation}
for each time $t \in [0,T]$ and apply the barrier term $\Theta(x(t), \alpha,  \varepsilon)$ to achieve (\(\varepsilon\)-relaxed) feasible solutions. For our numerical experiments we investigate a small directed acyclic time-dependent network flow problem with $4$ nodes and $4$ edges and consider two different time-dependent cost functions on the edges for the time interval $[0,1]$. This problem and the two different time-dependent cost functions, $c_{e}^q(t)$ for all edges $e\in E$, $t\in[0,1]$, and $q\in\{1,2\}$, are illustrated in Figure~\ref{fig:diamond}. 

\begin{figure}[htb!]
	\begin{minipage}{\textwidth}
		\centering
		\begin{tikzpicture}[scale=0.75,-latex,>=stealth,shorten >=1pt,auto,node distance=3cm, main node/.style={circle,draw,font=\small\footnotesize,minimum size=20pt,inner sep=2pt},
			every label/.style={font=\footnotesize}]
			\node[main node, label= left:{\footnotesize\(b_1=4\)}] (1) at (0,3) {1};
			\node[main node, label=above:{\footnotesize\(b_2=0\)}] (2) at (3,5.5) {2};
			\node[main node, label=below:{\footnotesize\(b_3=0\)}] (3) at (3,0.5) {3};
			\node[main node, label=right:{\footnotesize\(b_4=-4\)}] (4) at (6,3) {4};
			
			\path[thick,every node/.style={font=\footnotesize}]
			(1) edge node [above,sloped] {\(c_{(1,2)}^q(t)\)} (2)
			(1) edge node [below,sloped] {\(c_{(1,3)}^q(t)\)} (3)
			(2) edge node [above,sloped] {\(c_{(2,4)}^q(t)\)} (4)
			(3) edge node [below,sloped] {\(c_{(3,4)}^q(t)\)} (4);
		\end{tikzpicture} 
	\end{minipage}
	
	\vspace{2ex}
	\begin{minipage}{\textwidth}
		\small
		\centering
		\begin{tabular}{l@{\extracolsep{3em}}l}
			\toprule
			Linear cost function & Hat-type piecewise linear cost function\\
			Cost of edges for $t\in[0,1]$ & Cost of edges for $t\in[0,1]$ \\ \midrule
			$c_{(1,2)}^1(t)=c_{(2,4)}^1(t)=100(1+t)$  &   $c_{(1,2)}^2(t)=c_{(2,4)}^2(t)= \begin{cases}
				100(1+2\,t),&  t< 0.5,\\
				100(1+2(1-t)),&  t \geq 0.5 
			\end{cases}
			$     \\ \hline
			$c_{(1,3)}^1(t)=c_{(3,4)}^1(t)=100(2-t)$   &  $c_{(1,3)}^2(t)=c_{(3,4)}^2(t)=\begin{cases}
				100(2-2\,t),&  t< 0.5,\\
				100(2-2(1-t)),&  t \geq 0.5 
			\end{cases}$    \\ \bottomrule
		\end{tabular}
	\end{minipage}
	
	\caption{Time-dependent example problem -- minimum cost network flow problem with $4$ nodes and $4$ edges, where $\lb=0$ and $\ub=4$ for each edge. We consider two different time-dependent cost functions 
		$c_e^q=c_{(i,j)}^q:[0,1]\to \R$ with $q\in\{1,2\}$. For these simple examples the optimal cost ($1000$) can be computed explicitly.}
	\label{fig:diamond}
\end{figure}

\subsection{Linear cost function}
By construction, the upper path, i.e., from node $1$ to $4$ through node $2$, has smaller cost than the lower path, i.e., from node $1$ to $4$ through node $3$ for $t\in[0,\frac{1}{2})$, while for $t\in(\frac{1}{2},1]$ the lower path is preferable, and at $t=\frac{1}{2}$ both paths are indistinguishable w.r.t.\ the cost.  
For the barrier term $\Theta$ we normalize the costs as above and set the parameters as $\alpha_{0} = 1$ and $\varepsilon_{0}=0.001$ while applying the same geometric
decrease in $\alpha$ and $\varepsilon$ as in the previous section. For the dynamics \eqref{eq:cost_time} we choose $\lambda=0.001$, i.e. the penalization w.r.t.\ $\ux$ is small. For the optimization approach we allow a maximum of $50$ gradient steps and $20$ Armijo iterations with an initial step length of $1000$ while stopping the scheme if $\sum_{i=1}^4 |P_{\mathcal W_\text{ad}}(\nabla \hat\calJ(\ux))_i|<10^{-6}$ or if the maximum number of iterations is reached. For simplicity, we discretize the time interval $[0,1]$ by $1000$ equidistant time steps. This ensures that the discrete adjoint and state information is available at the same points in time.  As the initial flow we choose the flow that only uses the upper path, $x(0)=(x_{(1,2)}(0),x_{(1,3)}(0),x_{(2,4)}(0),x_{(3,4)}(0))^\top=(4,0,4,0)^\top$, i.e., the cost optimal path at time $t=0$, and expect that the optimizer redirects the flow over time to the
lower path $x=(0,4,0,4)^\top$ that is optimal at time $t=1$.
We actually observe this in our numerical tests.  After 25 gradient steps the optimizer updates the control $\ux$ such that, after starting in $x(0)=(4,0,4,0)^\top$, the flow at time $t=1$, i.e., after $1000$ time steps, has the value of $x(1)\approx(0.0213, 3.9787, 0.0213, 3.9787)^\top$ which lies in the precision of the applied Euler scheme and the barrier term $\Theta$. In Figure~\ref{fig:linDiamCost}, the total cost over time for each gradient step in Algorithm~\ref{alg:gradientdescent} is shown, where the red line at $1000$ depicts the optimal costs that is attained for a flow that switches from the upper path to the lower path at time $t=\frac{1}{2}$. 
Note that we can not expect to reach this optimum value, due to the smoothness constraints imposed on our control $\ux(t)$.
We observe a descent in total cost over time in $25$ gradient iterations until the optimizer can not find a feasible Armijo step length anymore, i.e., reaches the maximum number of Armijo iterations. We expect this behaviour due to the barrier approach.

The course of each component of the control $\ux$ after the first update and of the optimized
control $\overline{u}^x$ over time are shown in Figure~\ref{fig:linDiamw1stOT} and Figure~\ref{fig:linDiamwOptOT}, respectively. Note that we have a two-fold symmetry in these figures. The $\ux_i$ values for edges from the same path, i.e., edges $(1,2)$ and $(2,4)$ of the upper path (cyan, dotted) and edges $(1,3)$ and $(3,4)$ of the lower path (magenta), are identical and change with the same ratio, and $\ux_i$ values of different paths have the same absolute values but differing signs, i.e., edges from the upper path lose flow over time while edges from the lower path gain the same amount over time. One can observe that more and more flow is redirected via $\ux$ during the optimization process. Furthermore, this is also visible in the course of the optimal flow values over time, see Figure~\ref{fig:linDiamxOptOT}. This nicely captures the expected behaviour of the optimizer for this simple time-dependent network flow problem with a linear cost function.

\begin{figure}[htb!]
	\centering
	\subfloat[Control $\ux$ after the first update over time where the values w.r.t.\ the edges $(1,2)$ and $(2,4)$ are dotted and in cyan and the values w.r.t.\ the edges $(1,3)$ and $(3,4)$ are in magenta.\label{fig:linDiamw1stOT}]{\includegraphics[width=0.45\textwidth]{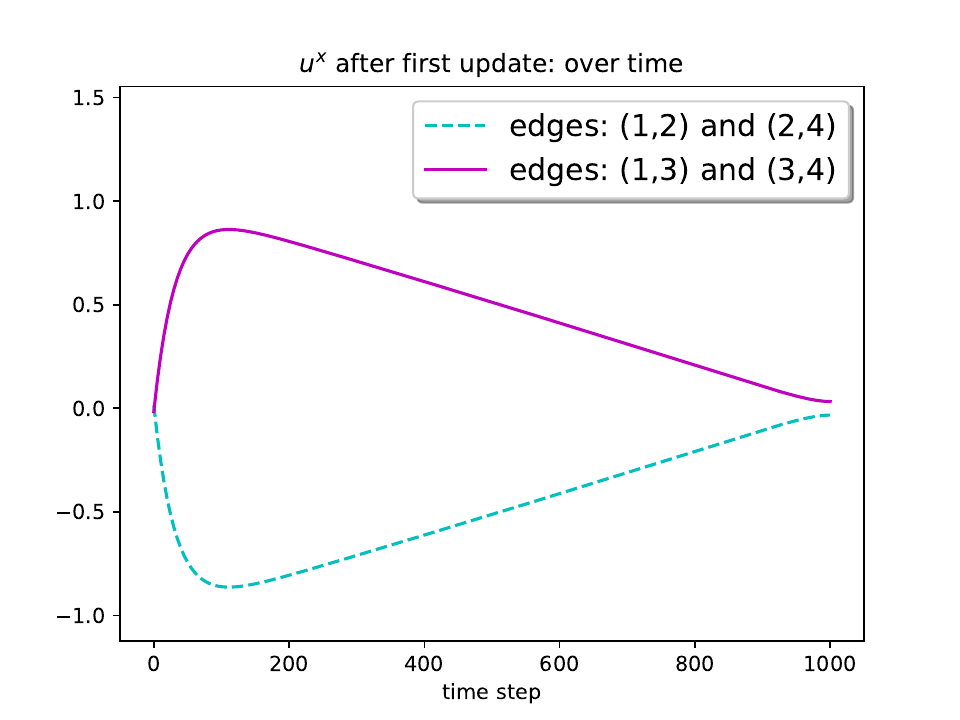}}
	\hspace{\fill}    
	\subfloat[Optimized control $\overline{u}^x$ over time. \label{fig:linDiamwOptOT}]{\includegraphics[width=0.45\textwidth]{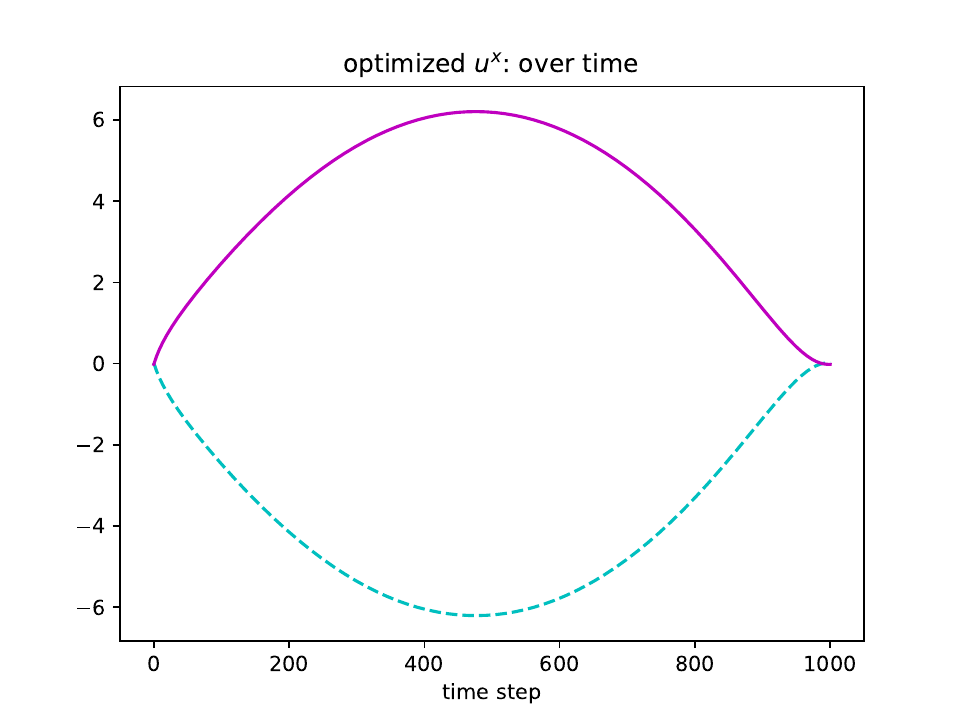}}
	
	\subfloat[Optimized flow values $\bar x$  over time. Here, the $\bar x$ components w.r.t.\ the edges $(1,2)$ and $(2,4)$ are dotted and in cyan and the values w.r.t.\ the edges $(1,3)$ and $(3,4)$ are in magenta. \label{fig:linDiamxOptOT}]{\includegraphics[width=0.45\textwidth]{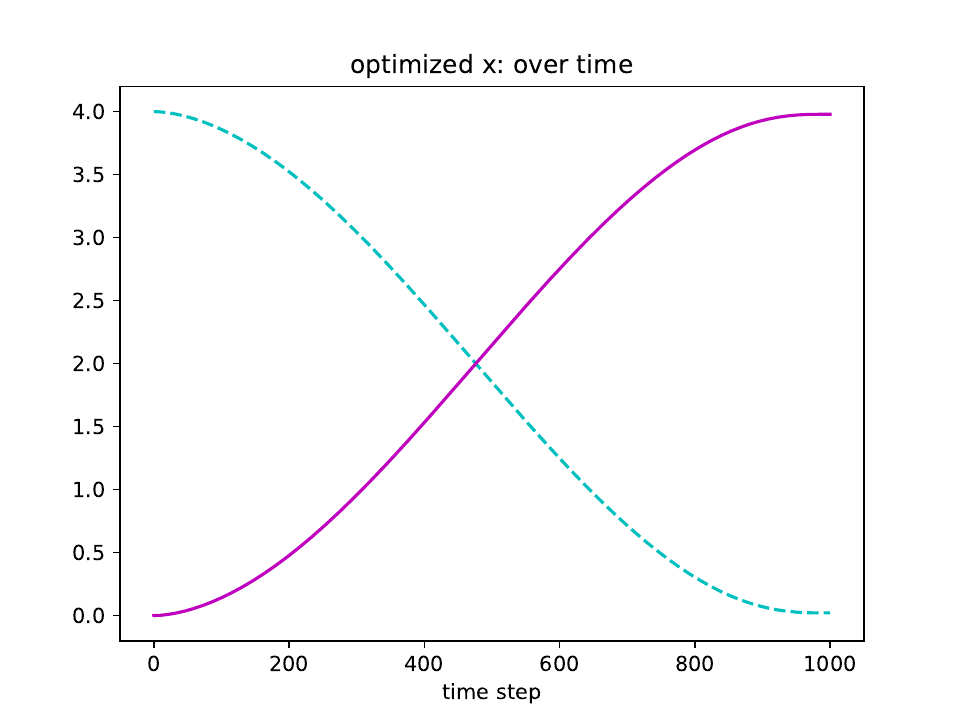}}
	\hspace{\fill}    
	\subfloat[Cost (over time) w.r.t.\ the gradient iterations. The red line depicts the optimal cost of $1000$.\label{fig:linDiamCost}]{\includegraphics[width=0.45\textwidth]{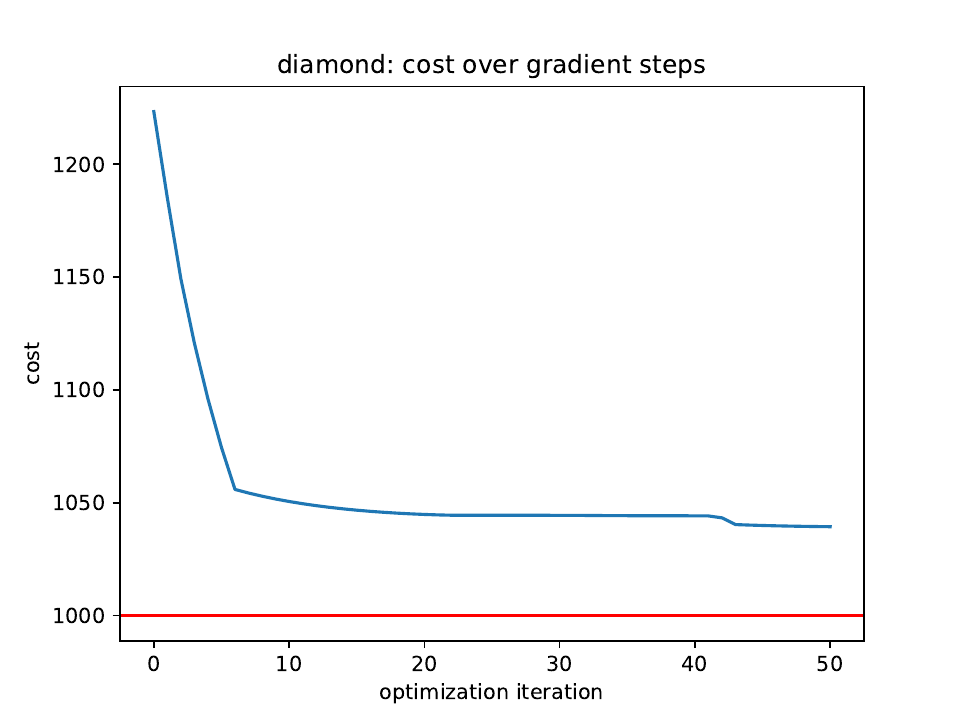}}
	\caption{Time-dependent problem with linear cost function}
	\label{fig:wOT}
\end{figure}

\subsection{Piecewise linear cost function}

\begin{figure}[htb!]
	\centering
	\subfloat[Control $\ux$ after the first update over time where the values w.r.t.\ the edges $(1,2)$ and $(2,4)$ are dotted and in cyan and the values w.r.t.\ the edges $(1,3)$ and $(3,4)$ are in magenta.\label{fig:hatDiamw1stOT}]{\includegraphics[width=0.45\textwidth]{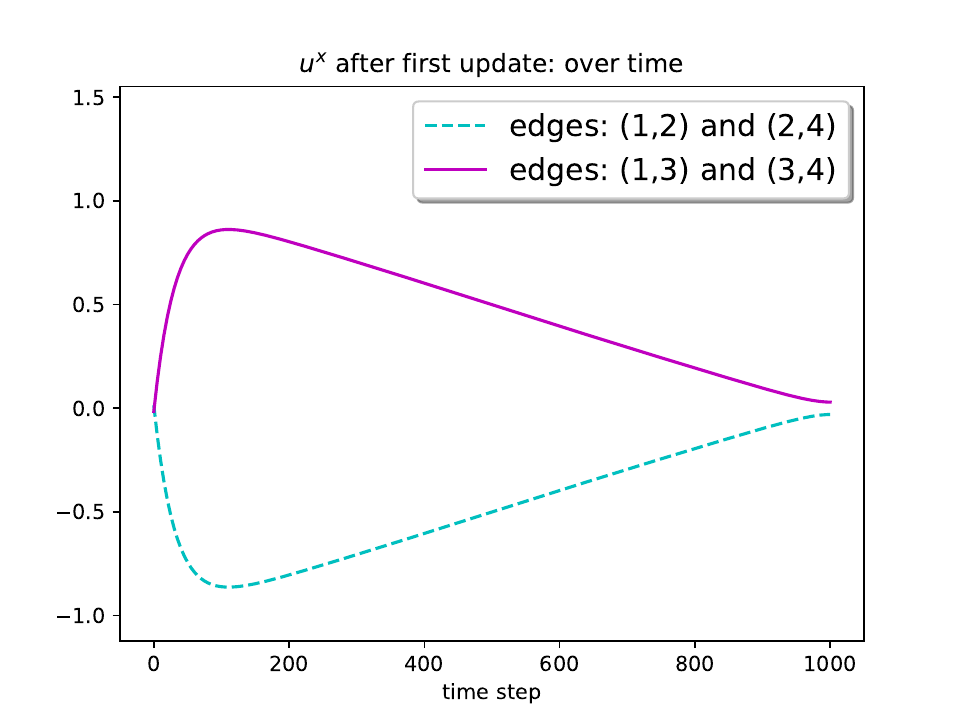}}
	\hspace{\fill}    
	\subfloat[Optimized control $\overline{u}^x$ over time. \label{fig:hatDiamwOptOT}]{\includegraphics[width=0.45\textwidth]{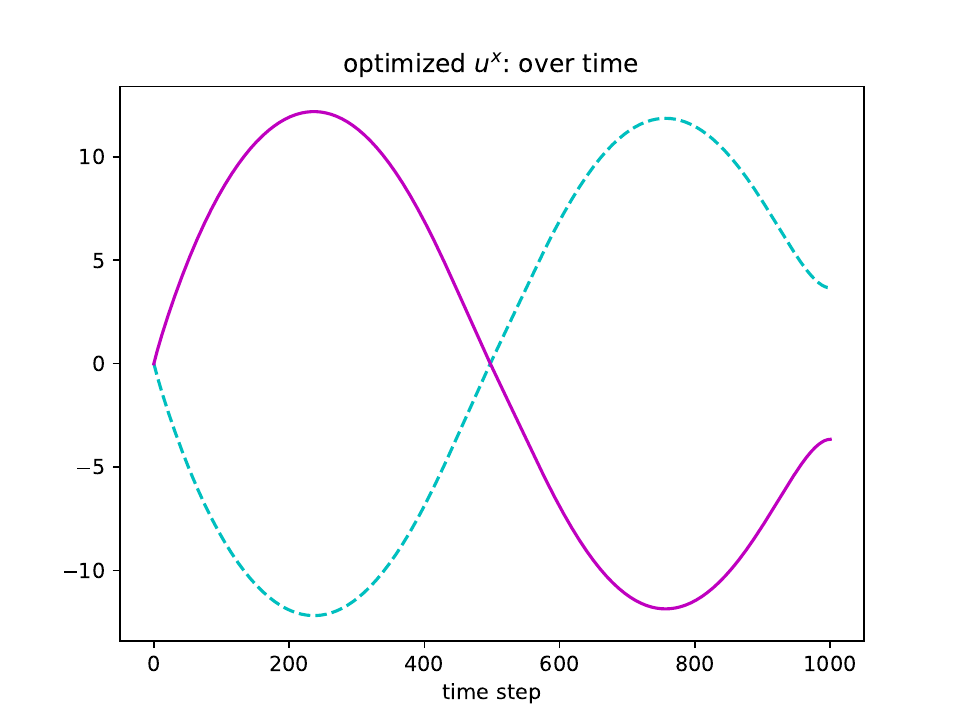}}
	
	\subfloat[Optimized flow values $\bar x$  over time. Here, the $\bar x$ components w.r.t.\ the edges $(1,2)$ and $(2,4)$ are dotted and in cyan and the values w.r.t.\ the edges $(1,3)$ and $(3,4)$ are in magenta. \label{fig:hatDiamxOptOT}]{\includegraphics[width=0.45\textwidth]{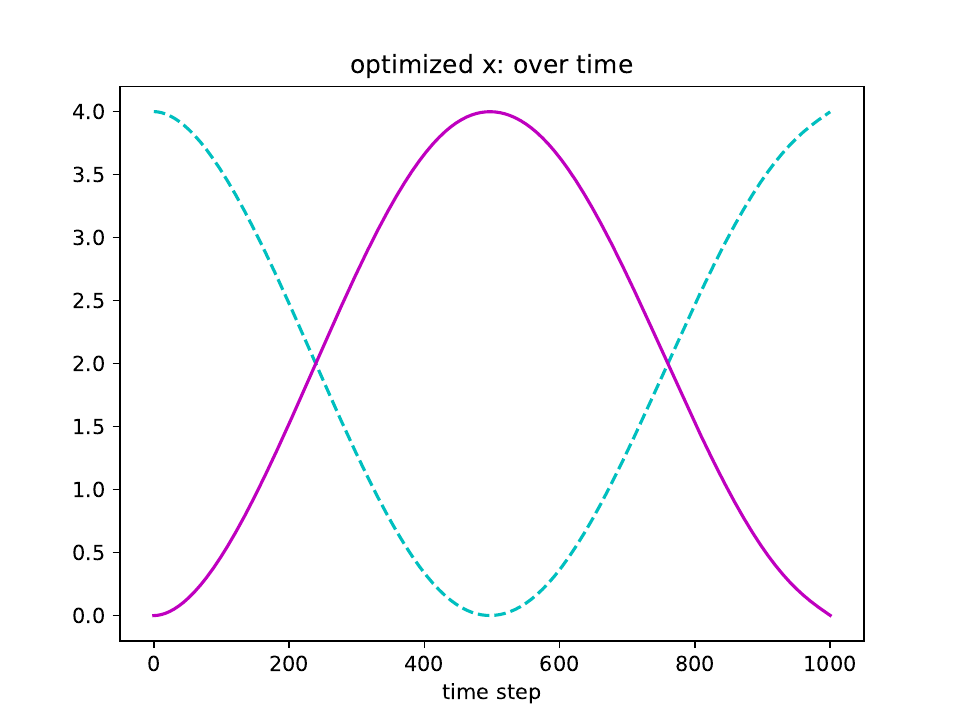}}
	\hspace{\fill}    
	\subfloat[Cost (over time) w.r.t.\ the gradient iterations. The red line depicts the non-continuous optimal cost of $1000$.\label{fig:hatDiamCost}]{\includegraphics[width=0.45\textwidth]{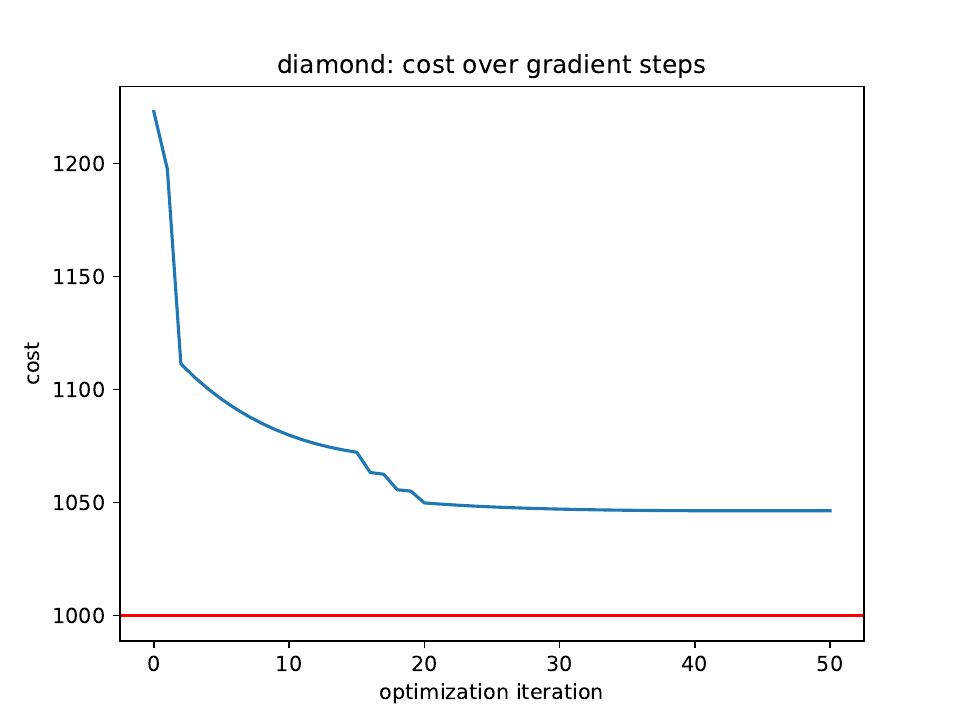}}
	\caption{Time-dependent problem with hat-type cost function}
	
\end{figure}

As a second example, we consider a piecewise linear cost function that resembles the shape of a hat, i.e., that is first linearly increasing and then linearly decreasing.  
This hat-type cost function, see Figure~\ref{fig:diamond}, is similar to the linear cost function in the sense that for $t\in[0, \frac{1}{4})$, the upper path has smaller cost than the lower path for $t\in[0,\frac{1}{4})$, at $t=\frac{1}{4}$ the paths are indistinguishable w.r.t.\ the cost, and for $t\in(\frac{1}{4},\frac{1}{2})$ the lower path is
preferable. However, for $t\in[\frac{1}{2}, 1]$ the roles reverse, i.e., the lower path is superior w.r.t.\ the cost for $t\in[\frac{1}{2}, \frac{3}{4})$, the paths are indistinguishable at $t=\frac{3}{4}$, and the upper path has again smaller cost for $t\in[\frac{3}{4}, 1]$. Hence, starting with the initial flow $x(0)=(x_{(1,2)}(0),x_{(1,3)}(0),x_{(2,4)}(0),x_{(3,4)}(0))^\top=(4,0,4,0)^\top$ that only uses the upper path one would expect that until $t=\frac{1}{2}$ the approach would shift the flow to the lower path, i.e., $x(\frac{1}{2})=(0,4,0,4)^\top$, and then redirect it back to the upper path for $t\in[\frac{1}{2}, 1]$, i.e., $x(1)=(4,0,4,0)^\top$. As for the linear cost function, the same setting for the barrier term $\Theta$ is applied, i.e., $\alpha_{0}= 1$ and $\varepsilon_{0}=0.001$ with the same geometric decrease. We also allow a maximum of $50$ gradient steps and $20$ Armijo iterations with the same initial step length of $1000$ and stopping condition $\sum_{i=1}^4 |P_{\mathcal W_\text{ad}}(\nabla \hat\calJ(\ux))_i|<10^{-6}$. Furthermore, the time interval $[0,1]$ is again discretized by $1000$ equidistant time steps. For the dynamics \eqref{eq:cost_time} we choose $\lambda=0.001$. The optimization approach used all $50$ gradient iterations, i.e., it did not converge, which we expect due to the barrier term. However, there are no significant changes to the cost after around $25$ iterations, see Figure~\ref{fig:hatDiamCost}. Note that the red line at $1000$ in Figure~\ref{fig:hatDiamCost} corresponds to the optimal costs.  
The control $\ux$ after the first iteration is depicted in Figure~\ref{fig:hatDiamw1stOT}. We observe that in this first iteration
the optimizer redirects flow from the upper path (cyan, dotted) to the lower path (magenta) without taking into account that after $t=\frac{1}{2}$, i.e., $500$ time steps, the flow should be redirected back to the upper path. This is not the case for the optimized control $\overline{u}^x$, where the control shifts the flow back to the upper path after initially redirecting it to the lower path, see Figure~\ref{fig:hatDiamwOptOT}. Moreover, the optimized flow values $\bar x$ for the times $t\in\{\frac{1}{2}, 1\}$ are as expected, i.e., $\bar x(\frac{1}{2})\approx(0.0011,3.9989,0.0011,3.9989)^\top$ and $\bar x(1)\approx(3.9984,0.0016,3.9984,0.0016)^\top$, see also Figure~\ref{fig:hatDiamxOptOT}. This nicely captures the expected behaviour for this hat-type cost functional.

To summarize, the presented approach yields satisfactory results for small, time-dependent network flow problems for both cost functions. 

\FloatBarrier

\section{Conclusion}
In this article we recast (static and dynamic) minimum cost flow problems in an optimal control framework for general nonlinear port-Hamiltonian systems of ODE-type.  The well-posedness of optimal control problems constrained by PHS is established and the first-order optimality system is derived. Based on this, a gradient descent algorithm exploiting the adjoint formulation is proposed. A special case of the discussed framework is given by (static and dynamic) minimum cost flow problems. Here the skew-symmetric matrix contains the information of the incidence matrix which defines the network at hand. We validate the approach with the help of simple static and dynamic network flow problems. 

In future work, we plan to extend the concept to more involved network flow problems and gain a deeper understanding of the relationship of optimal control for PHS on graphs and dynamic network flow problems. Intrinsic features of PHS, such as the decomposition into subnetworks will be exploited to speed up the algorithms. Moreover, dissipation of flow that may occur, for example, due to leaking pipes, can easily be integrated into the model in the PHS framework.

\section{Appendix}

\subsection{Proof of Theorem~\ref{thm:boundedness_state}}\label{app:boundedness_state}
Let $Q\in\R^{n\times n}$ such that $Q>0$, let $J,R \in \text{Lip}_\text{loc}(\R^n, \R^{n\times n})$ with $ J(z)^\top = -J(z)$ for all $z \in \R^n$, let $R(z) \ge 0$ for all $z \in \R^n$, let $B\in \text{Lip}_\text{loc}(\R^n, \R^{n\times m})$, and let $u \in L^2(0,T;\R^m)$.
We want to show that the state solution $z$ to \eqref{eq:phs} is bounded by the control $w$, i.e., there is a constant $C>0$ such that
$\| z \|_{H^1(0,T;\R^n)} \leq C \,\| w \|_W$.

To show this, we first find a bound on the $L^2$-norm of $z$. Note that the regularity of $z$ shown in Theorem~\ref{thm:existence_state} implies the embedding $z\in C([0,T],\R^n)$ and in particular yields the existence of a constant $C_0>0$ such that 
\[
\max\limits_{s\in [0,T]} \Bigl( |J(z(s))| + |R(z(s))| + |B(z(s))| \Bigr) \le C_0.
\]
Using the Young inequality and the Jensen inequality \cite{alt2016linear} we find
\begin{align*}
	|z(t)|^2 &= 2|\zin|^2 + 2 |\int_0^t \bigl(J(z(s)) - R(z(s)) \bigr)\, Q\, z(s) + B(z(s))\, u(s) \dd s |^2, \\
	&\le 2|\zin|^2 + 4 T \int_0^t |\bigl(J(z(s)) - R(z(s)) \bigr)\, Q\, z(s)|^2 + |B(z(s))\, u(s)|^2 \dd s \\
	&\le  2|\zin|^2 + 4 T C_0^2 \left( |Q|^2 \int_0^t |z(s)|^2 \dd s + \int_0^T |u(s)|^2 \dd s \right)
\end{align*}
for all $t\in[0,T]$. An application of the Gronwall inequality \cite{evans1998partial}
gives us
\[
|z(t)|^2 \le \alpha e^{4TC_0^2 |Q|^2 t}
\]
where 
\[\alpha = 2\, \,|\zin|^2 + 4\, T \, C_0^2 \int_0^T  |u(s)|^2 ds.
\]
Integration over $[0,T]$ yields the existence of $C_1>0$ such that 
\[
\| z \|_{L^2(0,T;\R^n)}^2 \leq C_1 \big( |\zin|^2 + \| u \|_{L^2(0,T;\R^m)}^2 \big) .
\]

Analogously, we obtain for the derivative 
\begin{align*}
	\Bigl| \frac{\dd}{\dd t }z(t) \Bigr|^2 
	&=  \bigl|\bigl(J(z) - R(z) \bigr)\, Q\, z + B(z)\, u(t) \bigr|^2 \\
	&\leq  2\, C_0^2 \, |Q|^2 \,|z(t)|^2 + 2\, C_0^2\, |u(t) |^2  \\
	&=  2\, C_0^2 \, |Q|^2 \Bigl|\int_0^t \frac{\dd}{\dd s}z(s) \dd s \Bigr|^2 + 2\, C_0^2\, |u(t)|^2   \\
	&\leq  2\, C_0^2 \, |Q|^2\, T\int_0^t \Bigl|\frac{\dd}{\dd s} z(s) \Bigr|^2 \dd s + 2\, C_0^2\, |u(t) |^2 . 
\end{align*}
Again, with Gronwall we obtain a constant $C_2>0$ such that
\[
\Bigl\| \frac{\dd}{\dd t} z \Bigr\|_{L^2(0,T;\R^n)}^2  \leq C_2 \, \| w \|_W^ 2. 
\]
Adding the two bounds and taking the square root on both sides  yields the desired result.

\subsection{Proof of Theorem~\ref{thm:optimal_input}}\label{app:optimal_input}
By assumption, $\lambda>0$, or $\Wad \subset W$ is closed and bounded. Moreover, $J, R, B$ are weakly continuous, $\cost(z,z_\text{des})$ is continuous and convex w.r.t.~$z$, and $\cost_T(z,z_\text{des})$ is continuous and convex w.r.t.~$z(T)$. 
We want to show that, under these assumptions, the optimal control problem \eqref{eq:opt_problem} has a solution.

To prove this, we follow the lines of \cite{hinze09opt,troeltzsch10optimal} and consider a minimizing sequence $\{w_n\}_{n\in \mathbb N} \subset \Wad$, that means $w_n = (u_n, \zin_n)$ with $$\inf\limits_{w} \hat \calJ(w) = \lim\limits_{n\rightarrow\infty} \hat \calJ(w_n).$$
Either $\lambda>0$ or the boundedness of $\Wad$ ensures the boundedness of $\{w_n\}_n.$ Note that $W$ is a reflexive Hilbert space, hence there exists a weakly convergent subsequence $\{w_{n_k}\}_{n_k}$ with $w_{n_k} \rightharpoonup \bar w \in W$ as $k\rightarrow \infty$. 
We denote $\bar w = (\bar u, \bar \zin)$ in the following. Theorem~\ref{thm:boundedness_state} yields the boundedness of $$\{ S(w_{n_k}) \} \subset H^1(0,T;\R^n)$$
and by reflexivity of the space we obtain the existence of $\bar z$ such that $S(w_{n_k}) =: z_{n_k} \rightharpoonup \bar z$ as $k\rightarrow \infty$. We emphasize that at this stage it is not clear that $\bar z$ is the state solution for the control $\bar w$. We prove this in the following. Indeed, the  weak continuity of $J,R$ and $B$ and the weak convergence of $z_{n_k}$ in $H^1(0,T;\R^n)$ allow us to obtain
\begin{align*}
	&\int_0^T \varphi(t) \, \frac{\dd}{\dd t} \big( z_{n_k}(t) - \bar z(t) \big) \dd t \longrightarrow 0, \\
	&\int_0^T \varphi(t)\, \bigl(J(z_{n_k}(t)) - J(\bar z(t)) - \bigl(R(z_{n_k}(t))- R(\bar z(t)) \bigr)\bigr)\, Q\, z_{n_k}(t) \dd t \longrightarrow 0, \\
	&\int_0^T \varphi(t)  \bigl(J(\bar z(t))  - R(\bar z(t)) \bigr) \bigl(z_{n_k}(t) - \bar z(t) \bigr) \dd t \longrightarrow 0,\\ 
	&\int_0^T \varphi(t)\, B(z_{n_k}(t))\, \bigl(w_{n_k}(t) - \bar w(t) \bigr) + \bigl(B(z_{n_k}(t)) - B(\bar z(t)) \bigr)\, \bar w(t) \dd t \longrightarrow 0,
\end{align*}
for all test functions $\varphi \in L^2(0,T;\R^n)$, which implies the weak continuity of the state operator $e$, i.e.~$e(z_{n_k},w_{n_k}) \rightharpoonup e(\bar z,\bar w)$. By the weak lower semi-continuity of the norm we find 
\begin{equation*}
	0\le \| e(\bar z, \bar w) \| = \liminf\limits_{k\rightarrow \infty} \| e(z_{n_k},w_{n_k})\|  = 0,
\end{equation*}
which proves that $\lim_{k\rightarrow\infty}S(w_{n_k}) = \lim_{k\rightarrow\infty} z_{n_k} = \bar z.$ 

We note that the continuity and convexity of $\cost(z,z_\text{des})$ and $\cost_T(z(T),z_\text{des}(T))$ imply the weak-lower semicontinuity of $\calJ$ and moreover of $\hat \calJ$. This allows to conclude that
\[
\hat\calJ(\bar w) \le \liminf\limits_{k\rightarrow\infty} \hat\calJ(w_{n_k}) = \inf\limits_{w} \hat \calJ(w),
\]
which proves that $\bar w$ is an optimal control.

\bibliographystyle{plain}
\bibliography{phS_networks}

\end{document}